\documentclass[10pt, a4paper, twoside]{amsart}

\usepackage[english]{babel}
\usepackage[latin1]{inputenc}

\usepackage{amsmath}
\usepackage{amsthm}
\usepackage{amssymb}
\usepackage{tikz}
\usepackage{csquotes}

\usepackage{imakeidx} 
\usepackage{appendix}
\usepackage{tikz-cd}
\usepackage{verbatim}
\usepackage{enumitem}
\usepackage{multicol}
\usepackage{textgreek}

\usepackage[backref=page]{hyperref}
\usepackage{cleveref}

\hypersetup{colorlinks=true,linkcolor=blue,anchorcolor=blue,citecolor=blue}

\usepackage{adjustbox}

\usepackage{mathtools}

\newtheorem{thm}{Theorem}[section]
\newtheorem{prop}[thm]{Proposition}
\newtheorem{lem}[thm]{Lemma}
\newtheorem{cor}[thm]{Corollary}

\theoremstyle{definition}
\newtheorem{df}[thm]{Definition}
\newtheorem{ex}[thm]{Example}

\theoremstyle{remark}
\newtheorem{rem}[thm]{Remark}
\numberwithin{equation}{section} 
\usepackage{stackengine,scalerel}
\stackMath
\newcommand\doublehat[1]{
  \stackon[-8.5pt]{\stackon[-7pt]{#1}{\widehat{\phantom{#1}}}}{\widehat{\phantom{#1}}}
}

\numberwithin{equation}{section}

\def\P{\mathbb{P}}
\def\C{\mathbb{C}}
\def\F{\mathbb{F}}
\def\Q{\mathbb{Q}}

\def\Z{\mathbb{Z}}
\def\R{\mathbb{R}}

\DeclareMathOperator{\Image}{Im}

\DeclareMathOperator{\Gal}{Gal}

\DeclareMathOperator{\Pic}{Pic}

\DeclareMathOperator{\Sym}{Sym}
\DeclareMathOperator{\id}{id}

\DeclareMathOperator{\Hom}{Hom}

\DeclareMathOperator{\ch}{ch}

\newcommand{\ov}{\overline}
\newcommand{\on}{\operatorname}
\newcommand{\mc}{\mathcal}

\author{Federico Scavia}
\address{CNRS\\
Institut Galil\'ee\\
	Universit\'e Sorbonne Paris Nord\\
	99 avenue Jean-Baptiste Cl\'ement, 93430\\ 
	Villetaneuse, France}
\email{scavia@math.univ-paris13.fr}

\author{Fumiaki Suzuki}

\address{Institute of Algebraic Geometry\\
Leibniz University Hannover\\
Welfengarten 1, 30167, Hannover\\
Germany
}
\email{suzuki@math.uni-hannover.de}

\title[\tiny On direct summands of products of Jacobians]{On direct summands of products of Jacobians over arbitrary fields}


\begin{document}

\begin{abstract}
We show that a principally polarized abelian variety over a field $k$ is, as an abelian variety, a direct summand of a product of Jacobians of curves which contain a $k$-point if and only if the polarization and the minimal class are both algebraic over $k$.
This extends results of Beckmann--de Gaay Fortman and Voisin over the complex numbers to arbitrary fields, and refines an obstruction to the direct summand property over $\mathbb{Q}$ due to Petrov--Skorobogatov.
We also give applications to the integral Tate conjecture for divisors and for $1$-cycles on abelian varieties over finitely generated fields; our results also address a $p$-adic version of the integral Tate conjecture over finite fields of characteristic $p$,
for the first time beyond the case of divisors.
\end{abstract}

\maketitle

\section{Introduction}

Let $X$ be a complex abelian variety. Is $X$ a direct summand of a product of Jacobians of smooth projective complex curves? 
This question has been related to questions about cycles
on $X$ by Beckmann--de Gaay Fortman \cite{beckmann2023integral} and Voisin \cite{voisin2022cycle}. We write $\widehat{X}$ for the dual abelian variety of $X$.

\begin{thm}[Beckmann--de Gaay Fortman, Voisin]\label{thm:BdFV}
Let \(X\) be a complex abelian variety of dimension \(g\), with Poincar\'e line bundle \(\mathcal{P}\). Assume that there exists \(\theta \in \on{NS}(X)\) such that \(\deg(\theta^g / g!) = \pm 1\). The following statements are equivalent:
\begin{enumerate}
\item The abelian variety \(X\) is a direct summand of a product of Jacobians. 
\item The class of $\ch(\mathcal{P})$ in $\on{Hdg}^{2*}(X \times \widehat{X}, \mathbb{Z})$ is algebraic.
\item The class of $\theta^{g-1}/(g-1)!$ in $\on{Hdg}^{2g-2}(X, \mathbb{Z})$ is algebraic. 
\item The integral Hodge conjecture for \(1\)-cycles holds for \(X\). That is, the cycle map \(CH^{g-1}(X) \to \on{Hdg}^{2g-2}(X, \mathbb{Z})\) is surjective.
\end{enumerate}
\end{thm}

\Cref{thm:BdFV} may be applied to any principal polarization $\theta\in NS(X)$,  or to the class $\theta\in NS(X\times\widehat{X})$ associated to the Poincar\'e line bundle $\mathcal{P}$.
Without assuming the existence of $\theta$, we only know (1)$\Leftrightarrow$(2)$\Rightarrow$(4).

The algebraicity of the class of $\ch(\mathcal{P})$ in $\on{Hdg}^{2*}(X\times \widehat{X},\Z)$ for every complex abelian variety $X$ is motivated by a question of Moonen--Polishchuk \cite{moonen2010divided} and Totaro \cite{totaro2021integral}, who asked whether the Fourier transform can be defined integrally on the Chow groups of abelian varieties over algebraically closed fields. As a corollary of \Cref{thm:BdFV}, the integral Hodge conjecture for \(1\)-cycles holds for all products of Jacobians, and it holds for all complex abelian varieties if and only if every complex abelian variety is a direct summand of a product of Jacobians. 
The question of whether properties (1), (2), (4) of \Cref{thm:BdFV} hold for every complex abelian variety $X$ is currently open; see the work of de Gaay Fortman--Schreieder \cite{dGF2024abelian} for a recent advancement.

Motivated by applications to the integral Tate conjecture ($\ell$-adic and crystalline) for abelian varieties, in this paper we set out to generalize \Cref{thm:BdFV} to an arbitrary ground field. In order to find the correct statement to generalize, an important point of departure for us was the recent construction by Petrov--Skorobogatov \cite{petrov2024spectral} of a principally polarized abelian surface $X$ over $\Q$ which is not a direct summand of a product of Jacobians of 
smooth projective $\Q$-curves which contain a $\Q$-point. Indeed, in the course of the proof, they showed that the polarization of $X$ is not represented by a line bundle on $X$. This suggested to us that the correct \emph{direct summand property} to consider, for an abelian variety $X$ over an arbitrary field $k$, should be \enquote{$X$ is a direct summand of a product of Jacobians of curves which contain a $k$-point}, and also that algebraicity of polarizations should come into the picture.

\subsection{Main result}
Let $k$ be a field, let $G\coloneqq \on{Aut}(\ov{k}/k)$, let $X$ be a smooth projective variety over a field $k$ of characteristic $p\geq 0$, and let $i\geq 0$ be an integer. For all primes $\ell\neq p$, we have a cycle map in $\ell$-adic cohomology
\begin{equation}\label{ell-adic-cycle-map}CH^i(X)_{\Z_\ell}\to H^{2i}(X_{\ov{k}}.\Z_\ell(i))^G.\end{equation}
When $p>0$, we also have a cycle map in crystalline cohomology
\begin{equation}\label{p-adic-cycle-map}CH^i(X)_{\Z_p}\to H^{2i}(X_{k_p}/W(k_p))^{F=p^i},\end{equation}
where $k_p$ is the perfect closure of $k$, $W(k_p)$ is the ring of ($p$-typical) Witt vectors of $k_p$, and the notation $F=p^{i}$ indicates the kernel of the endomorphism $F-p^{i}$, where $F$ is the absolute Frobenius endomorphism of $X$ (that is, the endomorphism of $X$ which is the identity on points and which raises regular functions to the $p$-th power). 
A cohomology class in the image of (\ref{ell-adic-cycle-map}) (when $\ell\neq p$) or in the image of (\ref{p-adic-cycle-map}) (when $\ell=p>0$) is said to be \emph{algebraic}. When $k$ is finitely generated over its prime field and perfect (that is, $k$ is either a finitely generated field of characteristic zero or it is finite), we say that {\it the integral Tate conjecture for codimension $i$-cycles holds for $X$} if the cycle maps (\ref{ell-adic-cycle-map}) and (\ref{p-adic-cycle-map}) are surjective for every prime number $\ell$.

Suppose that $X$ is an abelian variety over $k$. In this case, the integral $\ell$-adic cohomology of $X_{\ov{k}}$ and the crystalline cohomology of $X_{k_p}$ are torsion-free, and hence we have inclusions 
\[H^{2i}(X_{\ov{k}},\Z_\ell(i))\subset H^{2i}(X_{\ov{k}},\Q_\ell(i)),\qquad H^{2i}(X_{k_p}/W(k_p))\subset H^{2i}(X_{k_p}/W(k_p))[1/p].\]
For every line bundle $L$ on $X$, the Chern character of $L$ is given by
\[\ch(L)=\sum_{i\geq 0} c_1(L)^i/i!\in  CH^*(X)_{\Q}.\]
We say that $\ch(L)$ is \emph{integral} if it belongs to the image of the natural homomorphism $CH^*(X)\to CH^*(X)_{\Q}$. Although $\ch(L)$ is not necessarily integral, its classes in $\ell$-adic cohomology and crystalline cohomology are integral;
see \Cref{lem:p-adic-abelvar}(2) for the case of crystalline cohomology.

\begin{thm}\label{thm:main1}
Let $k$ be a field of characteristic $p\geq 0$, and let $X$ be a abelian variety over $k$, of dimension $g$ and with Poincar\'e line bundle $\mc{P}$. Assume that there exists $\theta\in \on{NS}(X_{\ov{k}})^G$ such that $\deg(\theta^g/g!)=\pm 1$. The following statements are equivalent:
\begin{enumerate}
\item[(1)] The abelian variety $X$ is a direct summand of a product of Jacobians of curves which contain a $k$-point.
\item[(1')] The abelian variety $X$ is a direct summand of a product of Jacobians of curves which contain a zero-cycle of degree $1$.
\item[(2)] 
The classes of $\ch(\mathcal{P})$ in $H^{2*}((X\times \widehat{X})_{\ov{k}},\Z_\ell(*))^G$ for all primes $\ell\neq p$ and 
in $H^{2*}((X\times \widehat{X})_{k_p}/W(k_p))^{F=p^*}$ for $\ell=p>0$ are algebraic.
If $k$ is finite, $\ch(\mathcal{P})$ is integral.
\item[(3)] The classes of $\theta^{g-1}/(g-1)!$ in 
$H^{2g-2}(X_{\ov{k}},\Z_\ell(g-1))^G$ for all primes $\ell\neq p$ and in $H^{2g-2}(X_{k_p}/W(k_p))^{F=p^{g-1}}$ for $\ell=p>0$ are algebraic.
\item[(3')] 
Let $i\in \{1,g-1\}$.
The classes of $\theta^i/i!$ in 
$H^{2i}(X_{\ov{k}},\Z_\ell(i))^G$ for all primes $\ell\neq p$
and in $H^{2i}(X_{k_p}/W(k_p))^{F=p^i}$ for $\ell=p>0$ are algebraic. 
\end{enumerate}
When $k$ is finitely generated and perfect, these statements are also equivalent to:
\begin{enumerate}
\item[(4)] The integral Tate conjecture for $1$-cycles holds for $X$. 
\item[(4')] The integral Tate conjecture for divisors and $1$-cycles holds for $X$. 
\end{enumerate}
\end{thm}

\Cref{thm:main1} may be applied to any principal polarization $\theta\in NS(X_{\ov{k}})^G$, or to the class $\theta\in NS((X\times \widehat{X})_{\ov{k}})^G$ associated to $\mathcal{P}$. Without assuming the existence of $\theta$, we prove that (1)$\Leftrightarrow$(1')$\Leftrightarrow$(2), and when $k$ is finitely generated and perfect, that (1)$\Rightarrow$(4')$\Rightarrow$(4). 
Observe that the analogues of (3') and (4') do not appear in \Cref{thm:BdFV}: indeed, by the Lefschetz $(1,1)$-theorem, they are equivalent to (3) and (4), respectively.

We will prove more precise versions of \Cref{thm:main1}, formulated for 
a given set $S$ of prime numbers $\ell$; see Theorems \ref{thm:main1-precise} and \ref{thm:main2-precise}. (One recovers \Cref{thm:main1} by letting $S$ be the set of all prime numbers.)

\Cref{thm:main1}(1)$\Leftrightarrow$(3')
(see also \Cref{lem:neron-severi})
shows that not only the non-algebraicity of a minimal class, but also that of a polarization obstructs the direct summand property. This refines the obstruction to the direct summand 
property over $\Q$ due to Petrov--Skorobogatov. Jacobians over global fields, whose canonical polarizations are not algebraic, have been classically studied; for instance, see the paper \cite{poonen1999cassels} of Poonen--Stoll. Combining \Cref{thm:main1} and \cite[Section 10]{poonen1999cassels}, one obtains many new examples of Jacobians $X$ over global fields $k$ (including $\F_q(t)$ with $q$ odd) of arbitrarily large dimensions, which are not direct summands of products of Jacobians of curves which contain a $k$-point. Since the obstruction to the algebraicity of a polarization of an abelian variety over a field is $2$-primary torsion (cf. \Cref{lem:neron-severi-sep}(3)), it follows that for such $X$ the class of 
$\theta^{g-1}/(g-1)!$ in $H^{2g-2}(X_{\ov{k}},\Z_2(g-1))^G$ is not algebraic.

Over finite fields, the equivalence \Cref{thm:main1}(1')$\Leftrightarrow$(2), together with the Lang--Weil estimate, shows that for a product $X$ of Jacobians  $\ch(\mathcal{P}) \in CH^*(X \times \widehat{X})_\Q$ is integral. Over infinite fields, the analogous statement may fail for a Jacobian of a curve which contains a rational point; see \cite[Example 3.1]{moonen2010divided}.

\subsection{Applications to the integral Tate conjecture for abelian varieties}

An important feature of \Cref{thm:main1} is that it addresses a $p$-adic version of the integral Tate conjecture over finite fields of characteristic $p$, for the first time beyond the codimension one case.
%As discussed above, our formulation is based on crystalline cohomology. \sca{Over finitely generated fields of characteristic $p$ and positive transcendence degree, the version of the integral Tate conjecture considered here would be insufficient for our purposes; see \cite[Corollary 6.7]{daddezio2024boundedness}.}
%; see \Cref{section:p-adic} for preliminaries on the crystalline cohomology of abelian varieties. 
%Marco D'Addezio has informed the authors that when the base field is imperfect, a more careful formulation would be necessary
%For divisors on abelian varieties, D'Addezio \cite{daddezio2024boundedness} studied similar problems.

%Over perfect finitely generated fields $k$, the integral Tate conjecture for divisors and $1$-cycles holds for products of Jacobians of curves which contain a $k$-point, according to \Cref{thm:main1}, whereas when $k$ is infinite, the conjecture for divisors (hence also for $1$-cycles) may fail already for Jacobians because polarizations on them are not necessarily algebraic, as mentioned above.

By a theorem of Tate \cite{tate1966endomorphisms}, the integral Tate conjecture for divisors holds for all abelian varieties over finite fields. From \Cref{thm:main1}, we deduce the following.

\begin{cor}\label{cor:main}
Let $\F_q$ be a finite field.
The integral Tate conjecture for $1$-cycles holds for all abelian varieties over $\F_q$ if and only if every abelian variety over $\F_q$ is a direct summand of a product of Jacobians.
\end{cor}

The validity of the integral Tate conjecture for $1$-cycles and of the direct summand property for all abelian varieties over $\F_q$ remain open. Using that 
the validity of the equivalent statements in \Cref{thm:main1} is invariant under Weil restrictions along finite separable extensions of the base field (see \Cref{lem:weil-restrictions-precise}), we show:

\begin{cor}\label{thm:tate-finite-wr-jacobians}
Let $\F_q$ be a finite field, let $\F_{q'}/\F_q$ be a finite extension, and let $X'$ be a product of Jacobians over $\F_{q'}$.
The Weil restriction $X\coloneqq \on{Res}_{\F_{q'}/\F_q}(X')$ is a direct summand of a product of Jacobians of curves which contain an $\F_q$-point.
In particular, $\ch(\mathcal{P})\in CH^{*}(X\times \widehat{X})_\Q$ is integral, and the integral Tate conjecture for $1$-cycles holds for $X$.
\end{cor}

\begin{cor}\label{thm:tate-finite-Fp}
If the integral Tate conjecture for $1$-cycles holds for all principally polarized abelian varieties over $\F_p$, then it holds for all abelian varieties over any finite field of characteristic $p$.
\end{cor}

Finally, using a result of Schoen \cite{schoen1998integral}, we prove:

\begin{cor}\label{thm:direct-summand-tate-consequence}
Assume that the Tate conjecture for divisors holds for all surfaces over every finite field of characteristic $p$.
Then, for every abelian variety over $\ov{\F}_p$, there exists an integer $n\geq 0$ such that the multiplication-by-$p^n$ map $(p^n)_X\colon X\to X$ factors through a product of Jacobians over $\ov{\F}_p$.
\end{cor}

\subsection{Structure of the paper}
The paper is organized as follows. \Cref{section:abelvar} contains preliminaries on N\'eron--Severi groups and crystalline cohomology of abelian varieties.
In \Cref{sec:algebraicity-minimal}, we prove \Cref{prop:geom-irr-rep}, which shows that every cycle on a smooth projective geometrically integral variety over a field is rationally equivalent to a finite sum of geometrically integral subvarieties;
we then use it to deduce \Cref{prop:minimal-class-direct-summand}, which generalizes \cite[Proposition 2.1]{voisin2022cycle}.
In \Cref{section:neron}, we prove \Cref{lem:neron-severi}, which shows that every polarization of a product of Jacobians of curves which contain a zero-cycle of degree $1$ is algebraic.
\Cref{section:main} is devoted to the proofs of \Cref{thm:main1-precise} and \Cref{thm:main2-precise}, which are more precise versions of \Cref{thm:main1}.
We also deduce \Cref{cor:main} and \Cref{thm:direct-summand-tate-consequence}. Finally, in \Cref{section:weil-restrictions}, we prove \Cref{lem:weil-restrictions-precise} on the behavior of the direct summand property under Weil restrictions, and then prove \Cref{thm:tate-finite-wr-jacobians} and \Cref{thm:tate-finite-Fp}.

\subsection{Notation}\label{notation}
Let $k$ be a field. We let $\ov{k}$ be an algebraic closure of $k$, and we let $k_s\subset \ov{k}$ and $k_p\subset \ov{k}$ be the separable closure and the perfect closure of $k$ inside $\ov{k}$, respectively. We set $G\coloneqq \on{Aut}(\ov{k}/k)$; recall that we have canonical isomorphisms of profinite groups $G\xrightarrow{\sim}\Gal(\ov{k}/k_p)$ and $G\xrightarrow{\sim}\Gal(k_s/k)$.

A $k$-variety is a separated $k$-scheme of finite type. Let $X$ be a smooth projective $k$-variety $X$. For every prime number $\ell$ invertible in $k$, we let $H^i(X_{\ov{k}},\Z_\ell(j))$ be the $\ell$-adic \'etale cohomology groups of $X_{\ov{k}}$. When $k$ is perfect of positive characteristic, we write $H^{*}(X/W(k))$ for the crystalline cohomology of $X$. We denote by $Z^i(X)$ the group of codimension $i$-cycles on $X$, by $CH^i(X)$ the Chow group of codimension $i$ cycles on $X$, by $CH^i(X)/\mathrm{hom}$ its quotient modulo the subgroup of homologically trivial cycles (i.e. the cycles that lie in the kernel of the cycle maps (\ref{ell-adic-cycle-map}) and (\ref{p-adic-cycle-map})), and by $CH^i(X)/\mathrm{num}$ its quotient modulo the subgroup of numerically trivial cycles (i.e. the cycles $\alpha\in CH^i(X)$ such that for all $\beta\in CH_i(X)$ we have $\deg(\alpha\cdot\beta)=0$).

A $k$-group is a group scheme locally of finite type over $k$. An abelian variety over $k$ is a proper (equivalently, projective) connected $k$-group. A Jacobian over $k$ is an abelian variety over $k$ which is isomorphic to the Jacobian of a smooth projective geometrically integral $k$-curve. For an abelian variety $X$ over $k$ and an integer $n\in \Z$, we write $n_X\colon X\rightarrow X$ for the multiplication-by-$n$ map.

For a $\Z$-algebra $R$ and a $\Z$-module $M$, we let $M_R\coloneqq M\otimes_{\Z}R$.
For every non-zero $n\in \Z$, we let $M[1/n]\coloneqq M_{\Z[1/n]}$,
and for a prime ideal $\mathfrak{p}$ of $\Z$, we let $M_{\mathfrak{p}}\coloneqq M_{\Z_{\mathfrak{p}}}$.

\section{Preliminaries on 
%N\'eron--Severi groups of 
abelian varieties}\label{section:abelvar}

\subsection{N\'eron--Severi groups}

Let $k$ be a field and let $X$ be an abelian variety over $k$. We have a short exact sequence of $k$-groups
\begin{equation}\label{pic-sequence}0\to \mathbf{Pic}_{X/k}^0\to \mathbf{Pic}_{X/k}\to \mathbf{NS}_{X/k}\to 0,\end{equation}
where $\mathbf{Pic}_{X/k}$ is the Picard scheme of $X$ (see \cite[Definition 4.1, Theorem 4.8]{kleiman2005picard}), the identity component $\mathbf{Pic}_{X/k}^0$ is an abelian variety, called the dual abelian variety $\widehat{X}$ of $X$ (see \cite[Remark 5.24]{kleiman2005picard}), and the fppf quotient $\mathbf{NS}_{X/k}$ is an \'etale $k$-group (see \cite[Expos\'e VIA, Proposition 5.5.1]{sga3I}). In particular, the $k$-group $\mathbf{Pic}_{X/k}$ is smooth. The formation of (\ref{pic-sequence}) commutes with arbitrary field extensions $k'/k$.

Since $X(k)\neq \emptyset$, the natural map $a\colon \on{Pic}(X)\to \mathbf{Pic}_{X/k}(k)=\on{Pic}(X_{k_s})^G$ is an isomorphism. We obtain a commutative diagram with exact rows 
\begin{equation}\label{pic-k-points}
\begin{tikzcd}
    0\arrow[r] & \on{Pic}^0(X) \arrow[r] \arrow[d,"\wr"] & \on{Pic}(X) \arrow[r]\arrow[d,"a"',"\wr"] & \on{NS}(X) \arrow[r] \arrow[d,hook] & 0 \\
     0\arrow[r] & \mathbf{Pic}_{X/k}^0(k) \arrow[r] & \mathbf{Pic}_{X/k}(k) \arrow[r] & \mathbf{NS}_{X/k}(k) \arrow[r] & H^1(k,\mathbf{Pic}_{X/k}^0), 
\end{tikzcd}
\end{equation}
where we set $\on{Pic}^0(X)\coloneqq a^{-1}(\mathbf{Pic}_{X/k}^0(k))$ and $\on{NS}(X)\coloneqq \on{Pic}(X)/\on{Pic}^0(X)$. Thus $\on{Pic}^0(X)\subset \on{Pic}(X)$ is the subgroup of algebraically trivial line bundles on $X$ and, when $k$ is algebraically closed, $\on{NS}(X)$ coincides with the classical N\'eron--Severi group of $X$.

For all abelian varieties $A$ and $B$ over $k$, the Hom functor $\mathbf{Hom}_k(A,B)$ is represented by a $k$-group scheme (see \cite[Corollary 4.2.4]{vandobben2018dominating}) which is locally of finite type (see \cite[Lemma 4.2.2]{vandobben2018dominating}), unramified (see \cite[Lemma 4.2.5]{vandobben2018dominating}), and hence \'etale (cf. \cite[Tag 02GU $(2)\Leftrightarrow(3)$]{stacks-project}). For every $k$-scheme $S$ and every line bundle $L$ on $X_S=X\times_kS$, we have a homomorphism $\varphi_L\colon X_S\to \widehat{X}_S$ defined as follows: for every $S$-scheme $S'$ and every $x\in X(S')$, we have $\varphi_L(x) = T_x^*(L_{S'})-L_{S'}$, where $L_{S'}$ is the base change of $L$ to $X_{S'}$ and $T_x\colon \widehat{X}_{S'}\to \widehat{X}_{S'}$ denotes the map given by translation by $x$. By sheafification, this construction determines a $k$-group homomorphism $\mathbf{Pic}_{X/k}\to \mathbf{Hom}(X,\widehat{X})$ given by $L\mapsto \varphi_L$. Since $\mathbf{Hom}(X,\widehat{X})$ is \'etale, this homomorphism is trivial on $\mathbf{Pic}_{X/k}^0$, and hence factors through a homomorphism of \'etale $k$-groups
\begin{equation}\label{phi-eq}\mathbf{NS}_{X/k}\to \mathbf{Hom}(X,\widehat{X}).
\end{equation}
The homomorphism (\ref{phi-eq}) is injective: for every $L\in \on{Pic}(X_{\ov{k}})$, the homomorphism $\varphi_L$ is equal to zero if and only if $L$ lies in $\on{Pic}^0(X_{\ov{k}})$; see \cite[p. 74]{mumford1970abelian}. (While in \emph{loc. cit.} it is assumed that $\on{char}(k)=0$, this is not necessary: the statement is a direct consequence of the Seesaw theorem \cite[Corollary 6 p. 54]{mumford1970abelian}, which holds in arbitrary characteristic.) We let \[\mathbf{Hom}^{\mathrm{sym}}(X,\widehat{X})\subset \mathbf{Hom}(X,\widehat{X})\] be the closed $k$-subgroup of $\mathbf{Hom}^{\mathrm{sym}}(X,\widehat{X})$ consisting of symmetric homomorphisms, that is, the homomorphisms $\varphi\colon X\to\widehat{X}$ such that $\varphi=\widehat{\varphi}\circ\kappa_X$, where $\kappa_X\colon X\xrightarrow{\sim} \doublehat{X}$ is the canonical isomorphism between $X$ and its double dual. The following result is standard.

\begin{lem}\label{lem-sym}
    Let $k$ be a field and let $X$ be an abelian variety over $k$. The image of the homomorphism (\ref{phi-eq}) is equal to $\mathbf{Hom}^{\mathrm{sym}}(X,\widehat{X})$.
\end{lem}

\begin{proof}
    We may assume that $k$ is algebraically closed. Furthermore, it is enough to prove the conclusion at the level of $k$-points. Let $\ell$ be a prime invertible in $k$, and let $T_\ell(X)\coloneqq \varprojlim_n X[\ell^n]$ be the $\ell$-adic Tate module of $X$. We have a natural $\Z_\ell$-bilinear pairing $e_\ell\colon T_\ell(X)\times T_\ell(\widehat{X})\to \Z_\ell(1)$; see \cite[p. 186]{mumford1970abelian}. For every homomorphism $\varphi\colon X\to \widehat{X}$, consider the $\Z_\ell$-bilinear pairing $E_\varphi\colon T_\ell(X)\times T_\ell(X)\to \Z_\ell(1)$ given by $E_\varphi(x,y)=e_\ell(x,T_\ell(\varphi)y)$. A direct computation shows that $\varphi$ is symmetric if and only if $E_\varphi$ is skew-symmetric. (For this, one must use the fact that, for all abelian varieties $A$ and $B$ over $k$, the natural map $\on{Hom}(A,B)_{\Z_\ell}\to \on{Hom}_{\Z_\ell}(T_\ell(A),T_\ell(B))$ is injective; see \cite[Theorem 3 p. 176]{mumford1970abelian}.)
    On the other hand, by \cite[Theorem 2 p. 188]{mumford1970abelian}, the pairing $E_\varphi$ is skew-symmetric if and only if $2\varphi=\varphi_L$ for some line bundle $L$ on $X$. By \cite[Remark p. 189]{mumford1970abelian}, the latter is equivalent to $\varphi=\varphi_{L'}$ for some line bundle $L'$ on $X$. In conclusion, $\varphi$ is symmetric if and only if $\varphi=\varphi_{L'}$ for some line bundle $L'$ on $X$, as desired.
\end{proof}

By \Cref{lem-sym}, the homomorphism (\ref{phi-eq}) induces an isomorphism
\begin{equation}\label{phi-eq-sym}\mathbf{NS}_{X/k}\xrightarrow{\sim}\mathbf{Hom}^{\mathrm{sym}}(X,\widehat{X}).
\end{equation}
 We also  let
\[\Hom^{\mathrm{sym}}(X,\widehat{X})\coloneqq \mathbf{Hom}^{\mathrm{sym}}(X,\widehat{X})(k)\subset \on{Hom}(X,\widehat{X})\]
be the image of (\ref{phi-eq}) at the level of $k$-points.

\begin{lem}\label{lem:neron-severi-sep}
Let $k$ be a field and let $X$ be an abelian variety over $k$. 

\begin{enumerate}
    \item The pullback map $\on{NS}(X_{k_s})\to \on{NS}(X_{\ov{k}})$ is an isomorphism.
    \item The pullback map $\mathbf{NS}_{X/k}(k)\to \mathbf{NS}_{X/k}(\ov{k})=\on{NS}(X_{\ov{k}})$ identifies $\mathbf{NS}_{X/k}(k)$ with $\on{NS}(X_{\ov{k}})^G$. In particular, taking $k$-points in (\ref{phi-eq-sym}) yields an isomorphism
    \begin{equation}\label{phi-eq-sym-k}
        \on{NS}(X_{\ov{k}})^G\xrightarrow{\sim} \on{Hom}^{\mathrm{sym}}(X,\widehat{X}),
    \end{equation}
    so that every element $\alpha\in \on{NS}(X_{\ov{k}})^G$ defines a symmetric homomorphism $\varphi_\alpha\colon X\to\widehat{X}$.
    \item The pullback map $\on{NS}(X)[1/2]\to \on{NS}(X_{\ov{k}})^G[1/2]$ is an isomorphism.
\end{enumerate}
\end{lem}

\begin{proof}
(1) Consider the commutative diagram
\[
\begin{tikzcd}
    \on{NS}(X_{k_s}) \arrow[r,"\sim"] \arrow[d]  & \mathbf{NS}_{X/k}(k_s) \arrow[d,"\wr"] & \\
    \on{NS}(X_{\ov{k}}) \arrow[r,"\sim"] & \mathbf{NS}_{X/k}(\ov{k}).
\end{tikzcd}
\]
Here, the left horizontal maps are the injections appearing in \eqref{pic-k-points} over the fields $k_s$ and $\ov{k}$; they are isomorphisms because the cohomology groups $H^1(k_s,\mathbf{Pic}_{X/k}^0)$ and $ H^1(\ov{k},\mathbf{Pic}_{X/k}^0)$ are trivial (for the first group, this follows from the smoothness of $\mathbf{Pic}_{X/k}^0$). This implies (1).

(2) For every $k$-scheme $Y$, by Galois descent the pullback map $Y(k)\to Y(k_s)^G$ is bijective. Thus the pullback map $\mathbf{NS}_{X/k}(k)\to \mathbf{NS}_{X/k}(k_s)^G$ is bijective. The first assertion is now a consequence of (1), and the second assertion follows from taking $k$-points in (\ref{phi-eq-sym}). 

We now prove (3). By (1), it is enough to show that the pullback map 
\begin{equation}\label{eq-ns-reduced}\on{NS}(X)[1/2]\to \on{NS}(X_{k_s})^G[1/2]\end{equation} is an isomorphism. The pullback $\on{NS}(X)[1/2]\to \on{NS}(X_{k_s})^G[1/2]$ is injective (cf. (\ref{pic-k-points})), and hence so is (\ref{eq-ns-reduced}). For the surjectivity of (\ref{eq-ns-reduced}), let $\alpha \in \on{NS}(X_{\ov{k}})^G$, and let $\varphi_\alpha\in \Hom^{\mathrm{sym}}(X,\widehat{X})$ be the image of $\alpha$ under (\ref{phi-eq-sym-k}). Letting $\delta\colon X\rightarrow X\times X$ be the diagonal map, we have
\[(\id\times \varphi_{\alpha})^*c_1(\mathcal{P})=m^*\alpha-\on{pr}_1^*\alpha-\on{pr}_2^*\alpha\quad \text{in $\on{NS}((X\times X)_{\ov{k}})$},\]
and hence
\[\delta^*(\id\times \varphi_{\alpha})^*c_1(\mathcal{P})=2\alpha\quad \text{in $\on{NS}(X_{\ov{k}})$}.\]
Thus $2\alpha$ comes from $\on{NS}(X)$. Thus (\ref{eq-ns-reduced}) is surjective, as desired.
\end{proof}

By definition, a \emph{polarization} on $X$ is an ample class in $\on{NS}(X_{\ov{k}})^G$. A polarization $\theta\in \on{NS}(X_{\ov{k}})^G$ is called \emph{principal} if the symmetric homomorphism $\varphi_\theta$ of (\ref{phi-eq-sym-k}) is an isomorphism. 

\begin{ex}\label{theta-divisor}
The following example of a N\'eron--Severi class will be especially important in the rest of the paper. Let $k$ be a field, let $C$ be a smooth projective geometrically integral curve over $k$, and suppose that $C$ admits a zero-cycle $\alpha$ of degree $1$. For every integer $d\geq 1$, we let $\Sym^d (C)$ be the $d$-symmetric power of $C$. We have a morphism 
\[\Sym^d (C)\rightarrow J(C),\qquad \beta\mapsto \beta-d \alpha, \] and we define \[W_d(C)\coloneqq \Image(\Sym^d (C)\rightarrow J(C)).\] The class of $W_{g-1}(C)$ in $\on{NS}(J(C)_{\ov{k}})^G$ is the canonical principal polarization $\theta_C$ of $J(C)$. More generally, for all $i\geq0$, the class of $W_i(C)$ in $(CH^{g-i}(X_{\ov{k}})/\on{num})_\Q$ is equal to $\theta_C^{g-i}/(g-i)!$; see \cite[Appendix]{matsusaka1959characterization} and also see \cite[Section 2]{mattuck1962symmetric}. 
\end{ex}

%\section{Preliminaries on the crystalline cohomology of abelian varieties}\label{section:p-adic}

\subsection{Fourier transform on crystalline cohomology 
%of abelian varieties
}\label{section:fourier-crystalline}

\begin{lem}\label{lem:p-adic-abelvar}
Let $k$ be a field of characteristic $p>0$, let $X$ be an abelian variety over $k$ of dimension $g$ with Poincar\'e bundle $\mc{P}$.
\begin{enumerate}
\item The natural $W(k_p)$-algebra map $\wedge^*H^1(X_{k_p}/W(k_p))\to H^*(X_{k_p}/W(k_p))$ is an isomorphism. Thus, for all $i\geq 0$ the $W(k_p)$-module $H^{2i}(X_{k_p}/W(k_p))$ is free of rank $\binom{2g}{i}$.
\item For every line bundle $L$ on $X$ and all $i\geq 0$, the class of $c_1(L)^i/i!$ in $H^{2i}(X_{k_p}/W(k_p))$ lies in $H^{2i}(X_{k_p}/W(k_p))^{F=p^i}$.
\item The homomorphism 
\[
\ch(\mathcal{P})_*\colon H^*(X_{k_p}/W(k_p))\rightarrow H^*(\widehat{X}_{k_p}/W(k_p))\] 
is an isomorphism and, for all $i,j\geq 0$, it restricts to an isomorphism
\begin{align*}
   (c_1(\mathcal{P})^{2g-i}/(2g-i)!)_*\colon& H^{i}(X_{k_p}/W(k_p))^{F=p^j}\xrightarrow{\sim} H^{2g-i}(\widehat{X}_{k_p}/W(k_p))^{F=p^{g-i+j}}.
\end{align*}
\end{enumerate}
\end{lem}
\begin{proof}
(1) The first statement and the fact that $H^1(X_{k_p}/W(k_p))$ is free of rank $2g$ are proved in \cite[Corollaire 2.5.5]{berthelot1982theorie}. The second statement now follows from the properties of the exterior power.
 
(2) Let $L$ be a line bundle on $X$. By (1), letting $e_1,\cdots, e_{2g}$ be a $W(k_p)$-basis  of $H^1(X_{k_p}/W(k_p))$, the class of $c_1(L)$ in $H^2(X_{k_p}/W(k_p))$ may be expressed as \[\sum_{1\leq n<m\leq 2g}c_{nm}\cdot e_n\wedge e_m\] for some $c_{nm}\in W(k_p)$.
Then, for all $i\geq 0$, the class of $c_1(L)^i$ in $H^{2i}(X_{k_p}/W(k_p))$ may be expressed as \[\sum_{1\leq n_1<\dots <n_{2i}\leq 2g}d_{n_1\dots n_{2i}}\cdot e_{n_1}\wedge\dots \wedge e_{n_{2i}}\] for some $d_{n_1\dots n_{2i}}\in W(k_p)$ divisible by $i!$. It follows that the class of $c_1(L)/i!$ belongs to $H^{2i}(X_{k_p}/W(k_p))^{F=p^i}$.

(3) By the K\"unneth decomposition for crystalline cohomology \cite[Theorem 5.1.2]{illusie1983finiteness}, we get
\[
H^{*}((X\times \widehat{X})_{k_p}/W(k_p))=H^*(X_{k_p}/W(k_p))\otimes H^*(\widehat{X}_{k_p}/W(k_p)).
\]
By \cite[Remarque II.3.11.2]{illusie1979complexe},
the $W(k_p)$-module $H^1(X_{k_p}/W(k_p))$ admits a natural structure of Dieudonn\'e crystal, i.e., an $F$-crystal with Verschiebung operator $V$.
By \cite[Th\'eor\`eme 5.1.8]{berthelot1982theorie}, the class of $c_1(\mathcal{P})$ in 
\[(H^1(X_{k_p}/W(k_p))\otimes H^1(\widehat{X}_{k_p}/W(k_p)))^{F=p}\subset  H^{2}((X\times \widehat{X})_{k_p}/W(k_p))^{F=p}
\]yields
an isomorphism of Dieudonn\'e crystals
\[H^1(\widehat{X}_{k_p}/W(k_p))\cong H^1(X_{k_p}/W(k_p))^{\vee},\] 
which induces an isomorphism of $\Z_p$-modules
\[
\resizebox{1\hsize}{!}{
$(H^1(X_{k_p}/W(k_p))\otimes H^1(\widehat{X}_{k_p}/W(k_p)))^{F=p}\cong\Hom_{\on{F-Crys}
(k_p)}(H^1(X_{k_p}/W(k_p)),H^1(X_{k_p}/W(k_p)))$.
}
\]
The rest of the proof follows from an argument due to Beauville (see the proof of \cite[Proposition 1]{beauville1983fourier}).
More precisely, choose a $W(k_p)$-basis $e_1,\dots, e_{2g}$ of $H^1(X_{k_p}/W(k_p))$. The class of $c_1(\mathcal{P})$ in $H^2((X\times \widehat{X})_{k_p}/W(k_p))^{F=p}$ may be expressed as $\sum_n e_n\otimes e_n^*$.
For every subset $I=\{n_{1}<\dots< n_{|I|}\}$ of $\{1,\dots,2g\}$, we set 
\[e_I:=e_{n_1}\wedge\dots\wedge e_{n_{|I|}},\qquad c(I)\coloneqq |I|(|I|-1)/2,\qquad d(I)\coloneqq g+|I|(|I|+1)/2,\]
and we define $\epsilon(I)\in\left\{\pm1\right\}$ so that $e_I\wedge e_{I^c}=\epsilon(I)e_1\wedge\dots\wedge e_{2g}$. A direct computation shows that
\[
\exp\left(\sum_n e_n\otimes e_n^*\right) =\prod_n \exp(e_n\otimes e_n^*)=\prod_n (1+e_n\otimes e_n^*)=\sum_{I\subset \left\{1,\dots,2g\right\}}(-1)^{c(I)}e_I\otimes e_I^*.
\]
Then $\ch(\mathcal{P})_*e_I = (-1)^{d(I)}\epsilon(I) e_{I^c}^*$ for all $I\subset \left\{1,\cdots, 2g\right\}$. This in turn shows that for all $i\geq 0$,
\[
\ch(\mathcal{P})_*=(c_1(\mathcal{P})^{2g-i}/(2g-i)!))_*\colon H^{i}(X_{k_p}/W(k_p))\xrightarrow{\sim}H^{2g-i}(\widehat{X}_{k_p}/W(k_p)).
\]
Restricting to suitable eigenspaces for $F$ proves (3). 
\end{proof}

\subsection{Integral crystalline Tate conjecture for divisors
%on abelian varieties
}\label{section:integral-crystalline-Tate}

\begin{thm}[Tate, Zarhin, Faltings]\label{thm:tate-zarhin-faltings-dejong}
Let $k$ be a finitely generated field of characteristic $p\geq 0$, and let $X$ be an abelian variety over $k$. For every prime number $\ell\neq p$, the direct limit of the $\ell$-adic cycle maps (\ref{ell-adic-cycle-map}) (for $i=1$) over all finite separable extensions $k'/k$
\[
CH^1(X_{k_s})\rightarrow H^2(X_{\ov{k}},\Z_\ell(1))^{(1)},
\]
where $H^2(X_{\ov{k}},\Z_\ell(1))^{(1)}$ is the $\Z_\ell$-submodule of $H^2(X_{\ov{k}},\Z_\ell(1))$ consisting of all elements with open $G$-stabilizer,
factors through an isomorphism 
\begin{equation}\label{eq:tate-zarhin-faltings-1}
NS(X_{k_s})_{\Z_\ell}\xrightarrow{\sim} H^2(X_{\ov{k}},\Z_\ell(1))^{(1)}.
\end{equation}

If $k=\F_q$ is finite, 
for every prime number $\ell\neq p$,
%(resp. $\ell=p$), 
%and $i=1$, 
the $\ell$-adic cycle map (\ref{ell-adic-cycle-map}) (for $i=1$) factors through an isomorphism
%in $\ell$-adic cohomology
\begin{equation}\label{eq:tate-zarhin-faltings-2}
NS(X)_{\Z_\ell}\xrightarrow{\sim} H^2(X_{\ov{\F}_q},\Z_\ell(1))^G,
\end{equation}
and for $\ell=p$, the crystalline cycle map (\ref{p-adic-cycle-map}) (for $i=1$) factors through an isomorphism
\begin{equation}\label{eq:tate-zarhin-faltings-3}
NS(X)_{\Z_p}\xrightarrow{\sim} H^2(X/W(\F_q))^{F=p}.
\end{equation}

%the $\ell$-adic (resp. crystalline) 
%cycle map (\ref{ell-adic-cycle-map}) (resp. (\ref{p-adic-cycle-map}))
%is surjective.
\end{thm}

%Here we denote by $H^2(X_{\ov{k}},\Z_\ell(1))^{(1)}$ the $\Z_\ell$-submodule of $H^2(X_{\ov{k}},\Z_\ell(1))$ consisting of all elements with open $G$-stabilizer.

\begin{proof}
Let $p$ be the characteristic of the base field.
For $\ell\neq p$, the result is a well-known consequence of theorems of Tate \cite{tate1966endomorphisms}, Zarhin \cite{zarhin1974isogenies,zarhin1974remark}, and Faltings \cite{faltings1983endlich,faltings1984erratum,faltings1986rational}.

Suppose now that $k=\F_q$ is finite.
The image of the map
\[ 
H^2(X/W(\F_q))\xrightarrow{m^*-\on{pr}_1^*-\on{pr}_2^*} H^{2}(X\times X/W(\F_q))
\]
is contained in $H^1(X/W(\F_q))\otimes H^1(X/W(\F_q))\subset H^{2}(X\times X/W(\F_q))$. We have the canonical identifications 
\begin{align*}
&(H^1(X/W(\F_q))\otimes H^1(X/W(\F_q)))^{F=p}\\
&\cong \Hom_{\on{F-Crys}
(\F_q)}(H^1(\widehat{X}/W(\F_q)),H^1(X/W(\F_q)))\\
&\cong \Hom(X\{p\},\widehat{X}\{p\}).
\end{align*}
Therefore we obtain a homomorphism
\begin{equation}\label{eq:daddezio}
H^2(X/W(\F_q))^{F=p}\rightarrow \Hom^{\mathrm{sym}}(X\{p\},\widehat{X}\{p\}).
\end{equation}
As explained in the proof of \cite[Proposition 4.6]{daddezio2024boundedness}, the map (\ref{eq:daddezio}) is an isomorphisms.
Consider now the commutative diagram
\begin{equation}\label{eq:tate-zarhin-faltings-dejong}
\begin{tikzcd}
\on{NS}(X)_{\Z_p} \arrow[r] \arrow[d, "\wr"] & H^2(X/W(\F_q))^{F=p} \arrow[d, "\wr"]\\
\Hom^{\mathrm{sym}}(X,\widehat{X})_{\Z_p}
\arrow[r]   &\Hom^{\mathrm{sym}}(X\{p\}, \widehat{X}\{p\}),
\end{tikzcd}
\end{equation}
where 
the left vertical map is the isomorphism (\ref{phi-eq-sym-k}) composed with the isomorphism $NS(X)\xrightarrow{\sim}NS(X_{\ov{\F}_q})^G$ due to Lang's theorem, the right vertical map is (\ref{eq:daddezio}), the top horizontal map is the cycle map appearing in the statement, and the bottom horizontal map is given by restriction to $p$-primary torsion subgroups. The commutativity follows from the definitions of the four maps and the isomorphisms
\[
(\id\times \varphi_L)^*\mathcal{P}\cong m^*L\otimes\on{pr}_1^*L^\vee\otimes\on{pr}_2^*L^\vee,
\]
which hold  for all line bundles $L$ on $X$. 

A theorem of Tate \cite[Theorem 7.1]{zarhin2008homomorphisms} shows that the natural map
\[
\Hom(X,\widehat{X})_{\Z_p}\rightarrow \Hom(X\{p\},\widehat{X}\{p\})
\]
is an isomorphism.
This implies that the bottom horizontal map in (\ref{eq:tate-zarhin-faltings-dejong}) is an isomorphism, and hence so is the top horizontal map, as desired.
\end{proof}

\section{Algebraicity of minimal classes and the direct summand property}\label{sec:algebraicity-minimal}

The goal of this section is to prove \Cref{prop:minimal-class-direct-summand}, which gives a sufficient condition for an abelian variety over a field $k$ to be a direct summand of a product of Jacobians of smooth projective curves $C_j$ such that $C_j(k)\neq\emptyset$. \Cref{prop:minimal-class-direct-summand} will be used in the proof of \Cref{thm:main1-precise}.

\begin{lem}\label{lem:geom-irr-rep-div}
Let $k$ be a field, let $X$ be a projective $k$-variety that is generically smooth and geometrically irreducible, let $d$ be the dimension of $X$, and let $D$ be a Cartier divisor on $X$.
\begin{enumerate}
    \item Assume that $k$ is infinite. There exist very ample effective Cartier divisors $H_1,H_2$ on $X$, which are generically smooth and geometrically irreducible, 
    such that 
     $[D]=[H_1]-[H_2]$ in $CH_{d-1}(X)$
     and $(H_i)_{\on{reg}}=X_{\on{reg}}\cap H_i$ for both $i=1,2$.
    \item Assume that $k$ is finite. There exist very ample effective Cartier divisors $H_1,H_2,H_3$ on $X$, which are generically smooth and geometrically irreducible, such that $[D]=[H_1]-[H_2]-[H_3]$ in $CH_{d-1}(X)$.
    \item  If $x\in X(k)$ is a smooth $k$-point, the $H_i$ in (1) and (2) may moreover be chosen to contain $x$ and to be smooth at $x$.
\end{enumerate}
\end{lem}
\begin{proof}
We fix a very ample divisor $H$ on $X$.

Suppose first that $k$ is infinite. By \cite[Example 1.2.10]{lazarsfeld2004positivity}, there exists an integer $n\geq 1$ such that $D+nH$ is very ample. By \cite[Theorem II.8.18]{hartshorne}, \cite[Th\'eor\`eme 6.3(4)]{jouanolou1983bertini}, and \cite[Theorem 3.4]{mainak2023bertini}, there exist geometrically smooth geometrically irreducible members $H_1\in |D+nH|$ and $H_2\in |nH|$ such that $[D]=[H_1]-[H_2]\in CH_{d-1}(X)$ and $(H_i)_{\on{reg}}=X_{\on{reg}}\cap H_{i}$ for both $i=1,2$. If $x\in X(k)$ is a smooth $k$-point, the same argument shows that the $H_i$ may be chosen to contain $x$ and to be smooth at $x$. This proves (1), as well as (3) when $k$ is infinite.

Suppose now that $k$ is finite. By \cite[Example 1.2.10]{lazarsfeld2004positivity}, \cite[Theorem 1.1]{poonen2004bertini}, and \cite[Theorem 1.1]{charles2016bertini}, there exists an integer $n\geq 1$ such that
$D+nH$ is very ample and 
such that there exists a generically smooth geometrically irreducible member $H_3\in |nH|$. By 
\cite[Theorem 1.1]{poonen2004bertini} and \cite[Theorem 1.1]{charles2016bertini}, there exists an integer $m\geq 1$ such that there exist generically smooth geometrically irreducible members $H_2\in |m(D+nH)|$ and  $H_1\in |(m+1)(D+nH)|$, so that $[D]=[H_1]-[H_2]-[H_3]\in CH_{d-1}(X)$. If $x\in X(k)$ is a smooth $k$-point, we run the same argument, replacing \cite[Theorem 1.1]{poonen2004bertini} by \cite[Theorem 3.3]{poonen2004bertini} and \cite[Theorem 1.1]{charles2016bertini} by \cite[Theorem 1.8]{charles2016bertini}, to conclude that the $H_i$ may be chosen to contain $x$ and to be smooth at $x$. This proves (2), as well as (3) when $k$ is finite.
\end{proof}

Let $k$ be a field and let $i$ be a positive integer $i$. We say that {\it resolution of singularities in dimension $i$ is available over $k$} if for every $i$-dimensional integral $k$-variety $X$, there exist a regular integral $k$-variety $Y$ and a birational proper morphism $Y\rightarrow X$. 
Resolution of singularities in dimension $i\leq 3$ is available over any field $k$; we will only need to use resolution of singularities in dimension $i=1$, which follows from normalization.

\begin{prop}\label{prop:geom-irr-rep}
Let $k$ be a field, let $X$ be a smooth projective geometrically integral $k$-variety, and let $i\geq 1$ be an integer. 
\begin{enumerate}
    \item Every $i$-cycle $\alpha$ on $X$ is rationally equivalent to an $i$-cycle of the form $\sum_j m_j Z_j$, where $Z_j\subset X$ is a geometrically integral closed subvariety of $X$ for each $j$.
    \item Assume that resolution of singularities in dimension $i$ is available over $k$. Then the $Z_j$ may be chosen as in (1) and with a birational morphism $W_j\rightarrow Z_j$, where $W_j$ is a smooth projective geometrically integral $k$-variety.
    \item Under the assumptions of (2), assume further that $X(k)\neq\emptyset$. Then the $Z_j$ may be chosen as in (2) and such that $Z_j$ has a smooth $k$-point for every $j$.
\end{enumerate}
\end{prop}

\begin{proof}
Let $d$ be the dimension of $X$. For each of (1) - (3), the conclusion is immediate when $i\geq d$. We may thus assume that $0<i<d$.

(1) It is enough to prove that every $i$-dimensional integral closed subvariety $Z$ of $X$ is rationally equivalent to an $i$-cycle $\sum_j m_j Z_j$ as in the statement.
Let $\pi\colon \widetilde{X}\rightarrow X$ be the blow-up of $X$ along $Z$ and let $E$ be the exceptional divisor of $\pi$.
By \cite[\S 4.3]{fulton1998intersection}, computing the $i$-dimensional component of the Segre class of $Z$ in $X$, we obtain \[[Z]=(-1)^{d-i-1}\pi_*([E]^{d-i})\quad \text{in $CH_i(X)$}.\]
The $k$-variety $\widetilde{X}$ is generically smooth because $X$ is smooth, and it is geometrically irreducible by \cite[Tag 0BFM]{stacks-project} and the geometric integrality of $X$. Since $E$ is a Cartier divisor on $\widetilde{X}$, 
\Cref{lem:geom-irr-rep-div} (1) (the first half) and (2) show that $[E]\in CH_{d-1}(\widetilde{X})$ is represented by a linear combination of generically smooth geometrically irreducible very ample effective Cartier divisors on $\widetilde{X}$.
Hence 
\[[E]^{d- i}=\sum_j \widetilde{m}_j [\widetilde{Z}_{j}]\quad \text{in $CH_i(X)$},\] where $\widetilde{Z}_j$ is a generically smooth geometrically irreducible complete intersection of very ample divisors on $\widetilde{X}$ for each $j$. For every subvariety $Y\subset \widetilde{X}$, a general member of a very ample divisor intersects properly with $Y$. Therefore, we may further assume that $\widetilde{Z}_j$ are distinct to each other and $\widetilde{Z}_j\not\subset E$, which implies that $\pi|_{\widetilde{Z_j}}$ are birational onto their distinct images. 
Now 
\[[Z]=(-1)^{d-i-1}\sum_j \widetilde{m}_j \pi_*[\widetilde{Z}_{j}]=(-1)^{d-i-1}\sum_j \widetilde{m}_j[Z_j]\quad \text{in $CH_i(X)$},\] 
where we define $Z_j\coloneqq \pi(\widetilde{Z}_j)\subset X$ with the reduced induced closed structure. Each $Z_j$ is reduced and, being birational to $\widetilde{Z}_j$, it is also geometrically irreducible, as well as geometrically reduced at its generic point. Now \cite[Tag 04KS (3)]{stacks-project} implies that $\widetilde{Z}_j$ is geometrically reduced, and hence geometrically integral.

(2) The statement is clear when $k$ is perfect. We may thus assume that $k$ is imperfect, and hence in particular infinite. Let $Z$ be an $i$-dimensional geometrically integral closed subvariety of $X$. When resolution of singularities in dimension $i$ is available over $k$, 
there exist a regular geometrically integral projective $k$-variety $W$ and a birational morphism $W\rightarrow Z$.
Embedding $W$ into $\P^N$ for $N$ large enough, the composite morphism $W\rightarrow Z\hookrightarrow X$ factors as $W\hookrightarrow{X\times \P^N}\rightarrow X$.
Hence, replacing $X$ by $X\times \P^N$ and $Z$ by $W$, we may assume that $Z$ is also regular.

For a $k$-variety $V$, we denote by $V_{\on{non-smooth}}\subset V$ the closed subset of points $v$ of $V$ such that the structure map $V\rightarrow \on{Spec}(k)$ is not smooth at $v$. Since $Z$ is geometrically integral, the codimension of $Z_{\on{non-smooth}}$ in $Z$ is at least $1$.

As in the proof of (1), let $\pi\colon \widetilde{X}\rightarrow X$ be the blow-up of $X$ along $Z$ and let $E$ be the exceptional divisor of $\pi$. Then $\widetilde{X}$ is regular, and it is smooth away from $(\pi|_E)^{-1}(Z_{\on{non-smooth}})$. 
Morever, since $Z\hookrightarrow X$ is a regular embedding, $\pi|_E\colon E\rightarrow Z$ is a projective bundle, which implies that the codimension of $(\pi|_E)^{-1}(Z_{\on{non-smooth}})$ in $E$ is at least $1$, so that the codimension of $(\pi|_E)^{-1}(Z_{\on{non-smooth}})$ in $\widetilde{X}$ is at least $2$. The argument of (1), combined with \Cref{lem:geom-irr-rep-div} (1) (the part about regular loci), then yields complete intersections $\widetilde{Z}_j\subset \widetilde{X}$ such that 
\begin{itemize}
    \item[(i)] $\widetilde{Z}_j$ is regular and geometrically integral, and
    \item[(ii)] the codimension of $(\widetilde{Z}_j)_{\on{non-smooth}}$ in $\widetilde{Z}_j$ is at least $2$ for each $j$.
\end{itemize}
  In other words,  we may further assume that there exists $W\rightarrow Z$ such that the codimension of $W_{\on{non-smooth}}$ in $W$ is at least $2$, and after replacing $X$ by $X\times \P^N$ and $Z$ by W, that the codimension of $Z_{\on{non-smooth}}$ in $Z$ is at least $2$. 
We repeat the above procedure to achieve $Z$ that is smooth.

(3) Assume further that $X$ has a smooth $k$-point $x$. By the moving lemma \cite[Tag 0B0D]{stacks-project}, it is enough to consider an $i$-dimensional integral closed subvariety $Z\subset X$ not containing $x$. We may repeat the same argument as in (2), except that this time we apply \Cref{lem:geom-irr-rep-div} (3) to $E$ and to a lift of $x$ in $\widetilde{X}$.
\end{proof}

\begin{rem}
    Each of the conclusions of \Cref{prop:geom-irr-rep} is false for $i=0$: a smooth projective geometrically integral $k$-variety might not have any $k$-point.
\end{rem}

\begin{prop}\label{prop:minimal-class-direct-summand}
Let $k$ be a field, let $g\geq 0$ be an integer, and let $X$ be an abelian variety of dimension $g$ over $k$.
Let $S$ be a set of prime numbers and $R\coloneqq \cap_{\ell\in S}\Z_{(\ell)}$.
Assume that there exists $\theta\in \on{NS}(X_{\ov{k}})^G_R$ such that
\begin{itemize}
    \item[(i)] $\deg(\theta^g/g!)$ is prime to $S$, and
    \item[(ii)] the class of $\theta^{g-1}/(g-1)!$ in $(CH^{g-1}(X_{\ov{k}})/\mathrm{num})_{\Q}$ belongs to the image of $CH^{g-1}(X)_R\rightarrow (CH^{g-1}(X_{\ov{k}})/\on{num})_\Q$.
\end{itemize}
Then there exist smooth projective geometrically integral $k$-curves $C_1,\dots,C_n$ such that $C_j(k)\neq\emptyset$ for all $1\leq j\leq n$ and a non-zero integer $N$ that is not divisible by any prime in $S$ such that the multiplication-by-$N$ map $N_X\colon X\rightarrow X$ factors through the product $\prod_j J(C_j)$.
\end{prop}
\begin{proof}
Without loss of generality, we may assume that $\theta\in \on{NS}(X_{\ov{k}})^G$.
We let $\alpha \in CH^{g-1}(X)_R$ be a lift of $\theta^{g-1}/(g-1)!\in (CH^{g-1}(X_{\ov{k}})/\mathrm{num})_\Q$.
There exists a non-zero integer $n$ that is prime to $S$ such that $n\cdot \alpha \in CH^{g-1}(X)$.
By \Cref{prop:geom-irr-rep}(3) applied to $i=1$, there exist a positive integer $n$ and, for all $1\leq j\leq n$, a smooth projective geometrically integral curve $C_j$ of genus $g_j$ such that $C_j(k)\neq\emptyset$, a morphism $f_j\colon C_j\rightarrow X$, and a coefficient $\epsilon_j \in \left\{\pm 1\right\}$, such that 
\[n\cdot \alpha = \sum_{j} \epsilon_j(f_j)_*[C_j]\quad \text{in $CH^{g-1}(X)$}.
\]
Let $Y\coloneqq \prod_j J(C_j)$
and $\theta_{\epsilon}\coloneqq  \sum_j \epsilon_j \on{pr}_j^*\theta_{C_j}\in \on{NS}(Y_{\ov{k}})^G$, where $\theta_{C_j}\in \on{NS}(J(C_j)_{\ov{k}})^G$ are the theta divisors; see \Cref{theta-divisor}. Let $\overline{f}_j\colon J(C_j)\rightarrow X$ be the morphisms of abelian varieties induced by $f_j$ and let $\overline{f}\coloneqq \sum_j \overline{f}_j\colon Y\rightarrow X$.
Consider the composite
\begin{equation}\label{eq:composition}
X\xrightarrow{\varphi_\theta} 
\widehat{X}\xrightarrow{\widehat{\overline{f}}} \widehat{Y}\xrightarrow{\substack{\varphi_{\theta_{\epsilon}}^{-1}\\\sim}}
Y
\xrightarrow{\overline{f}} X,
\end{equation}
where $\varphi_{\theta_{\epsilon}}$ is an isomorphism because $\theta_{\epsilon}$ is unimodular, that is, the degree of $\theta_{\epsilon}^{\sum_j g_j}/(\sum_jg_j)!$ is equal to $\pm 1$; see \cite[The Riemann--Roch Theorem p. 150]{mumford1970abelian}. 
To finish the proof, it remains to show that the composite \eqref{eq:composition} equals $(n\cdot m)_X$, where $m\coloneqq \deg(\theta^g/g!)$. To establish this, we adapt the argument given in \cite[Section 4]{achter2023prym}
and use the endomorphisms $\delta(\beta,\gamma)$ associated to cycles $\beta,\gamma$ defined by Matsusaka \cite[\S 4.2 and \S 4.3]{achter2023prym}.
More precisely, we have
\begin{equation}\label{eq:delta-1}
\delta(n\cdot \theta^{g-1}/(g-1)!, \theta)= -n\cdot \deg(\theta^{g}/g!)\cdot \id = - (n\cdot m)_{X}
\end{equation}
by \cite[Proposition 4.3 (c)]{achter2023prym}
and
\begin{equation}\label{eq:delta-2}
\delta(n\cdot \theta^{g-1}/(g-1)!, \theta)=\sum_j\epsilon_j\delta((f_j)_*C_j, \theta)=-\sum_j\epsilon_j \overline{f}_j\varphi_{\theta_{C_j}}^{-1}\widehat{\overline{f}}_j\varphi_\theta
\end{equation}
by \cite[Lemma 3.2, Proposition 4.4]{achter2023prym}.
The conclusion follows from the combination of \eqref{eq:delta-1} and \eqref{eq:delta-2}.
\end{proof}

\section{N\'eron--Severi group, the Chern character of the Poincar\'e line bundle, and the direct summand property}\label{section:neron}

For an abelian variety $X$ over a field $k$, the pullback map
$\on{NS}(X)\rightarrow \on{NS}(X_{\ov{k}})^G$ is not necessarily surjective; see \cite[Propositions 26 - 30]{poonen1999cassels}.
Nevertheless, the following holds:

\begin{prop}\label{lem:neron-severi}
Let $k$ be a field and let $X$ be an abelian variety over $k$. Suppose that there exists an odd integer $n$ such that the multiplication-by-$n$-map $n_X\colon X\rightarrow X$ factors through a product of Jacobians of curves which contain a zero-cycle of degree $1$. Then the pullback map $\on{NS}(X)\to \on{NS}(X_{\ov{k}})^G$ is an isomorphism.
\end{prop}

\begin{proof}
We will prove that the pullback map $\on{NS}(X)\to \on{NS}(X_{k_s})^G$ is an isomorphism; combined with \Cref{lem:neron-severi-sep}, this will imply (2). The map $\on{NS}(X)\to \on{NS}(X_{k_s})^G$ is injective by (\ref{pic-k-points}). We prove its surjectivity in three steps.

Suppose first that $X=J(C)$ is the Jacobian of a smooth projective geometrically integral $k$-curve $C$ which admits a zero-cycle $\alpha$ of degree $1$. Letting $f\colon C\rightarrow X$ be the morphism determined by $\alpha$, pullback along $f$ yields a commutative diagram with exact rows
\begin{equation}\label{diag:c-jc}
\begin{tikzcd}
    0\arrow[r] & \mathbf{Pic}_{X/k}^0 \arrow[r] \arrow[d,"\wr"] & \mathbf{Pic}_{X/k} \arrow[r]\arrow[d,"f^*"] & \mathbf{NS}_{X/k} \arrow[r] \arrow[d] & 0 \\
     0\arrow[r] & \mathbf{Pic}_{C/k}^0 \arrow[r] & \mathbf{Pic}_{C/k} \arrow[r] & \mathbf{NS}_{C/k} \arrow[r] & 0, 
\end{tikzcd}
\end{equation}
where the left vertical map is an isomorphism by \cite[Lemma 3.1]{achter2023prym}.
It follows that the top row of (\ref{diag:c-jc}) is the pullback of the bottom row. The group scheme $\mathbf{NS}_{C/k}$ is isomorphic to the constant group scheme $\Z$, and the map $\mathbf{Pic}_{C/k}\to \Z$ is the degree map. Since $C$ admits a zero-cycle of degree $1$, the bottom row of (\ref{diag:c-jc}) is split, and hence so is the top row. In view of (\ref{pic-k-points}), we conclude that the map $\on{NS}(X)\to \on{NS}(X_{k_s})^G$ is an isomorphism.

Suppose now that $X = \prod_{i=1}^nJ(C_i)$, where $C_1,\dots, C_n$ are smooth projective geometrically integral $k$-curves, each with a zero-cycle of degree $1$. Then
\[
\on{NS}(X)=\left(\bigoplus_{i=1}^n\on{NS}(J(C_i))\right)\oplus \left(\bigoplus_{1\leq i<j\leq n}\Hom(J(C_i),\widehat{J(C_j)})\right).
\]
By the previous part of the proof, the pullback maps  $\on{NS}(J(C_i))\to\on{NS}(J(C_i)_{k_s})^G$ are isomorphisms for all $1\leq i\leq n$. By Galois descent, the pullback maps $\Hom(J(C_i),\widehat{J(C_j)})\to\Hom(J(C_i)_{k_s},\widehat{J(C_j)}_{k_s})^G$ are isomorphisms, for all $1\leq i<j\leq n$. We conclude that the pullback map $\on{NS}(X)\to\on{NS}(X_{k_s})^G$ is an isomorphism.

Finally, consider the general case: there exists an odd integer $n$ such that $n_X\colon X\rightarrow X$ factors as $X\xrightarrow{f}Y\xrightarrow{h}X$, where $Y$ is a product of Jacobians of curves which contain a zero-cycle of degree $1$.
Let $\beta\in \on{NS}(X_{k_s})^G$.
By the previous part of the proof, there exists $\gamma\in \on{NS}(Y)$ such that $h^*\beta =\gamma_{k_s}$ in $\on{NS}(Y_{k_s})^G$.
Then $n^2\cdot \beta = f^*h^*\beta = (f^*\gamma)_{k_s}$ in $\on{NS}(X_{k_s})^G$.
On the other hand, by \Cref{lem:neron-severi-sep}(3), there exists a positive integer $m$ such that the element $2^m\cdot \beta\in \on{NS}(X_{k_s})^G$ comes from $\on{NS}(X)$.
Since $n^2$ and $2^m$ are coprime, we conclude that $\beta$ comes from $\on{NS}(X)$.
The proof is complete.
\end{proof}

\begin{prop}\label{prop:fourier-jacobians}
Let $k$ be a field of characteristic $p\geq 0$,
and let $X$ be a product of Jacobians of curves which contain a zero-cycle of degree $1$ over $k$.
\begin{enumerate}
    \item For every prime $\ell$, the class of $\ch(\mathcal{P})$ in $\ell$-adic cohomology $(\ell\neq p)$ and in crystalline cohomology $(\ell=p>0)$ is algebraic.
    \item Suppose that $k$ is finite. Then $\ch(\mathcal{P})\in CH^*(X\times\widehat{X})_{\Q}$ is integral.
\end{enumerate} 
\end{prop}

\begin{proof}
(1) Let $X=\prod_{i=1}^nJ(C_i)$, where the $C_i$ are smooth projective geometrically integral $k$-curves which contain a zero-cycle of degree $1$. Then $\mathcal{P}_X=\bigotimes_i \on{pr}_i^*\mathcal{P}_{J(C_i)}$. By the multiplicativity of the Chern character, we are reduced to the case $n=1$, that is, we may assume that $X=J(C)$, where  $C$ is a smooth projective geometrically integral $k$-curve which contains a zero-cycle of degree $1$.
Let $\theta_C\in NS(X_{\ov{k}})^{G}$ be the theta divisor; see \Cref{theta-divisor}.
The pullback map $NS(X)\to NS(X_{\ov{k}})^G$ is an isomorphism by \Cref{lem:neron-severi}, hence $\theta_C$ descends to a class in $NS(X)$, which we also denote by $\theta_C$.

By \cite[Section 4.1]{achter2023prym}, for every abelian variety $Y$ over $k$, numerical equivalence and homological equivalence coincide for $1$-cycles on $Y$, that is, the quotient map
\[CH_1(Y)/\mathrm{hom}\to CH_1(Y)/\mathrm{num}\]
is an isomorphism. 
Combined with \cite[Appendix]{matsusaka1959characterization} (see also \cite[Section 2]{mattuck1962symmetric}), we deduce that the map
\[CH_1(X) \to (CH_1(X)/\mathrm{hom})_{\Q}\]
sends $[W_1(C)]$ to $\theta_C^{g-1}/(g-1)!$. 
In particular, for every prime number $\ell\neq p$, the map
\[CH^{g-1}(X)_{\Z_\ell} \to H^{2g-2}(X_{\ov{k}},\Z_\ell(g-1))^G\]
sends $[W_1(C)]$ to the class of $\theta_C^{g-1}/(g-1)!$,
and if $p>0$, the same is true for the map
\[
CH^{g-1}(X)_{\Z_p}\rightarrow H^{2g-2}(X_{k_p}/W(k_p))^{F=p^{g-1}}.
\]
One then deduces from \cite[Proposition 3.11 2$\Leftrightarrow$5]{beckmann2023integral} that 
for every prime number $\ell\neq p$,
the class of $\ch(\mathcal{P})$ in $H^{2*}((X\times \widehat{X})_{\ov{k}},\Z_\ell(*))^G$ is algebraic.
If $p>0$, an argument analogous to the proof of \cite[Proposition 3.11]{beckmann2023integral} using \cite[Lemma 3.11]{beckmann2023integral} shows that the class of $\ch(\mathcal{P})$ in $H^{2*}((X\times \widehat{X})_{k_p}/W(k_p))^{F=p^*}$ is algebraic.

(2) When $k$ is finite, by \cite[Th\'eor\`eme 4, i)]{soule1984groupes} and \cite[Section 4.1]{achter2023prym}, the natural quotient map
\[CH_1(X)_{\Q}\to (CH_1(X)/\mathrm{num})_{\Q},\]
is an isomorphism, and hence $\theta_C^{g-1}/(g-1)!\in CH^{g-1}(X)_\Q$ lifts to the integral cycle $[W_1(C)]\in CH^{g-1}(X)$.

At this point, if we knew that some lift of $\theta_C$ in $CH^1(X)$ is symmetric, i.e. that it is invariant under the action of $(-1)_{X}$ on $CH^1(X)$, the proof would be complete by \cite[Theorem 3.8, 2$\Leftrightarrow$5]{beckmann2023integral}. It is however unlikely that this is always the case. On the other hand, for every abelian variety $Y$ over $k$, the group $\Pic^0(Y)=\hat{Y}(k)$ is finite, and hence by (\ref{pic-k-points}) the quotient map
\[CH^1(Y)_{\Q}\to \on{NS}(Y)_{\Q}\]
is an isomorphism. Since any class in $\on{NS}(Y)$ is symmetric by \cite[\S 8, (iv)]{mumford1970abelian}, we deduce that any class in $CH^1(Y)_{\Q}$ is symmetric, that is, it is invariant under the action of $(-1)_Y$ on $CH^1(Y)_{\Q}$. In particular, the class $\theta_C\in CH^1(X)_{\Q}$ is symmetric.

Identify $X$ and $\widehat{X}$ by $\varphi_{\theta_C}\colon X\xrightarrow{\sim}\widehat{X}$, let $\delta \colon X\rightarrow X\times X$ be the diagonal map, and for $i=1,2$ let $s_i\colon X\rightarrow X\times X$ be the inclusion of $X$ in $X\times X$ as the $i$-th factor. Define 
\[\Theta\coloneqq \frac{1}{2}\delta^*c_1(\mathcal{P}).\]
By \cite[\S  6, Corollary 3]{mumford1970abelian} and the fact that the class $\theta_C\in CH^1(X)_{\Q}$ is symmetric, we have the following string of equalities in $CH^1(X)_{\Q}$:
\[\Theta =\frac{1}{2}\delta^*(m^*- \on{pr}_1^*- \on{pr}_2^*)\theta_C =\frac{1}{2}((2_{X})^*\theta_C - 2\theta_C) =\theta_C.\]
Therefore, by \cite[Lemma 3.5]{beckmann2023integral}, we have
\begin{align*}
    c_1(\mathcal{P})^{2g-1}/(2g-1)! &=(-1)^{g+1}((s_1)_*+(s_2)_*-\delta_*)(\Theta^{g-1}/(g-1)!) \\
    &=(-1)^{g+1}((s_1)_*+(s_2)_*-\delta_*)(\theta_C^{g-1}/(g-1)!)
\end{align*}
in $CH^{2g-1}(X\times X)_{\Q}$.
This shows that $c_1(\mathcal{P})^{2g-1}/(2g-1)!\in CH^{2g-1}(X\times X)_{\Q}$ lifts to 
\[
\alpha\coloneqq 
(-1)^{g+1}((s_1)_*+(s_2)_*-\delta_*)[W_1(C)]\in CH^{2g-1}(X\times X).\]
Finally, for every integer $i\geq 0$, let $(-)^{[i]}$ be the $i$-th divided power operation on $CH_{>0}(X\times X)$, as defined by Moonen--Polishchuk \cite{moonen2010divided} (see also \cite[Theorem 3.7]{beckmann2023integral}). Since $c_1(\mathcal{P})^{2g-1}/(2g-1)!$ lifts to $\alpha$, by \cite[Lemma 3.4]{beckmann2023integral}, the class $\ch(\mathcal{P})\in CH^*(X\times X)_{\Q}$ lifts to \[
(-1)^{g}\sum_{i\geq 0}((-1)^g\alpha)^{[i]}\in
CH^*(X\times X),\]
and hence it is integral, as desired.
\end{proof}

\section{Proof of Theorem \ref{thm:main1}}\label{section:main}

The goal of this section is to prove \Cref{thm:main1}. In fact, we will deduce \Cref{thm:main1} from the more general Theorems \ref{thm:main1-precise} and \ref{thm:main2-precise}. We will also prove Corollaries \ref{cor:main} and \ref{thm:direct-summand-tate-consequence}.

\begin{df}
Let $k$ be a field, let $X, Y$ be abelian varieties over $k$, and let $S$ be a set of prime numbers.
We say that {\it $X$ is a prime-to-$S$ direct summand of $Y$} if there exists a non-zero integer $n$ that is not divisible by any prime in $S$ such that the multiplication-by-$n$ map $n_X\colon X\rightarrow X$ factors through $Y$.
\end{df}

Given a non-zero rational number $u=a/b$, where $a$ and $b$ are coprime integers, and a set of primes $S$, we say that $u$ is \emph{prime to $S$} if $a$ and $b$ are not divisible by any prime of $S$. Equivalently, $u$ is prime to $S$ if there exists an integer $m$, not divisible by any prime in $S$, such that $um$ is an integer not divisible by any prime in $S$.

\begin{thm}\label{thm:main1-precise}
Let $k$ be a field of characteristic $p\geq 0$, let $X$ be an abelian variety over a field $k$, of dimension $g$ and with Poincar\'e line bundle $\mathcal{P}$.
Let $S$ be a set of prime numbers and $R\coloneqq \cap_{\ell\in S}\Z_{(\ell)}$.
The following are equivalent:
\begin{enumerate}
\item The abelian variety $X$ is a  prime-to-$S$ direct summand of a product of Jacobians of curves which contain a $k$-point.
\item The abelian variety $X$ is a  prime-to-$S$ direct summand of a product of Jacobians of curves which contain a zero-cycle of degree $1$.
\item The classes of $\ch(\mathcal{P})$ in $H^{2*}((X\times \widehat{X})_{\ov{k}},\Z_\ell(*))^G$ for all $\ell\in S\setminus \{p\}$ and in $H^{2*}(X_{k_p}/W(k_p))^{F=p^*}$ for $\ell=p\in S$ are algebraic. If $k$ is finite, $\ch(\mathcal{P})$ is $S$-integral, that is, it belongs to the image of the homomorphism $CH^*(X\times\widehat{X})_R\to CH^*(X\times\widehat{X})_\Q$. 
\end{enumerate}
If there exists $\theta\in \on{NS}(X_{\ov{k}})^G_R$ such that $\deg(\theta^g/g!)$ is prime to $S$, then (1) - (3) are also equivalent to the following: 
\begin{enumerate}
\setcounter{enumi}{3}
\item 
Let $i\in \{1,g-1\}$.
The classes of $\theta^{i}/i!$ in $H^{2i}(X_{\ov{k}},\Z_\ell(i))^G$ for all $\ell\in S\setminus \{p\}$ and in $H^{2i}(X_{k_p}/W_{k_p})^{F=p^i}$ for $\ell=p\in S$ are algebraic.
\item The classes of $\theta^{g-1}/(g-1)!$ in 
$H^{2g-2}(X_{\ov{k}},\Z_\ell(g-1))^G$ for all $\ell\in S\setminus \{p\}$ and in $H^{2g-2}(X_{k_p}/W(k_p))^{F=p^{g-1}}$ for $\ell=p\in S$ are algebraic.
\end{enumerate}
\end{thm}

\begin{proof}
$(1) \Longrightarrow (2)$. Obvious.

$(2)\Longrightarrow(3)$. Let $Y$ be a product of Jacobians over $k$ and let $n$ be an integer not divisible by any prime in $S$ such that $n_X\colon X\rightarrow X$ factors as $X\xrightarrow{f} Y\xrightarrow{h} X$. By \cite[Theorem p. 125]{mumford1970abelian}, we have an isomorphism \[(\id \times \widehat{h})^*\mathcal{P}_Y\cong(h\times \id)^*\mathcal{P}_X\] of line bundles on $Y\times \widehat{X}$. Therefore
\begin{equation}\label{eq:poincare}
(n_{X\times \widehat{X}})_*(\id \times n_{\widehat{X}})^*(f\times \widehat{h})^*\ch(\mathcal{P}_Y)=(n_{X\times \widehat{X}})_*(n_{X\times \widehat{X}})^*\ch(\mathcal{P}_X)=n^{4g}\cdot \ch(\mathcal{P}_X)
\end{equation}
in $CH^*(X\times \widehat{X})_{\Q}$. Since by \Cref{prop:fourier-jacobians} the classes of $\ch(\mathcal{P}_Y)$ in integral $\ell$-adic (resp. crystalline) cohomology for all primes $\ell\neq p$ (resp. for $\ell=p>0$) are algebraic, \eqref{eq:poincare} shows that the same is true for the classes of $\ch(\mathcal{P}_X)$ for all $\ell \in S$.
When $k$ is finite, $\ch(\mathcal{P}_Y)\in CH^*(Y\times \widehat{Y})_{\Q}$ is integral by \Cref{prop:fourier-jacobians}. Therefore $n^{4g}\cdot\ch(\mathcal{P}_X)\in CH^{*}(X\times \widehat{X})_{\Q}$ lifts to $CH^*(X\times \widehat{X})$.

$(1)\Longrightarrow (4)$. By \Cref{lem:neron-severi}, the pullback map $\on{NS}(X)_R\to \on{NS}(X_{\ov{k}})^G_R$ is an isomorphism. 
In particular, the class of $\theta$ in integral $\ell$-adic (resp. crystalline) cohomology for all primes $\ell\in S\setminus \{p\}$ (resp. for $\ell=p\in S$) are algebraic. Moreover, since (1) implies (3), 
the classes of $c_1(\mathcal{P})^{2g-2}/(2g-2)!$ in integral $\ell$-adic (resp. crystalline) cohomology for all primes $\ell\in S\setminus \{p\}$ (resp. for $\ell=p\in S$) are algebraic. The map $\varphi_\theta$ induces isomorphisms
\[(\varphi_\theta)_*\colon H^{2g-2}(X_{\ov{k}},\Z_\ell(g-1))\xrightarrow{\sim}H^{2g-2}(\widehat{X}_{\ov{k}},\Z_\ell(g-1))\]
for all $\ell \in S\setminus \{\ell\}$, and if $p\in S$, an isomorphism
\[
(\varphi_\theta)_*\colon H^{2g-2}(X_{k_p}/W(k_p))^{F=p^{g-1}}\xrightarrow{\sim}H^{2g-2}(\widehat{X}_{k_p}/W(k_p))^{F=p^{g-1}},
\]
where in any of these cases
\[
((\varphi_\theta)_*)^{-1}=(1/\deg\varphi_\theta)\cdot (\varphi_\theta)^*=(1/\deg(\theta^g/g!)^2)\cdot (\varphi_\theta)^*
\]
by \cite[The Riemann--Roch Theorem p. 150]{mumford1970abelian};
combined with \cite[Proposition 5]{beauville1983fourier}, we have
\begin{align*}
\theta^{g-1}/(g-1)!& = (-1)^{g-1}\cdot (\deg(\theta^g/g!))\cdot ((\varphi_\theta)_*)^{-1}(c_1(\mathcal{P})^{2g-2}/(2g-2)!)_*\theta\\
&=(-1)^{g-1}\cdot(1/\deg(\theta^g/g!))\cdot (\varphi_\theta)^*(c_1(\mathcal{P})^{2g-2}/(2g-2)!))_*\theta
\end{align*}
in integral $\ell$-adic (resp. crystalline) cohomology for all primes $\ell\in S\setminus \{p\}$ (resp. for $\ell=p\in S$). In particular, the classes of $\theta^{g-1}/(g-1)$ in integral $\ell$-adic (resp. crystalline) cohomology for all primes $\ell\in S\setminus \{p\}$ (resp. for $\ell=p\in S$) are algebraic,
as desired.

$(4)\Longrightarrow (5)$. Obvious.

$(5)\Longrightarrow(1)$. By \Cref{prop:minimal-class-direct-summand}, it is enough to show that the class of $\theta^{g-1}/(g-1)!$ in  $(CH^{g-1}(X_{\ov{k}})/\mathrm{num})_{\Q}$ belongs to the image of the natural map \[CH^{g-1}(X)_R\to (CH^{g-1}(X_{\ov{k}})/\mathrm{num})_{\Q}.\]
We will prove the following stronger assertion: The class of $\theta^{g-1}/(g-1)!$ in $(CH^{g-1}(X_{\ov{k}})/\mathrm{num})_{\Q}$ belongs to the subgroup $(CH^{g-1}(X)/\mathrm{num})_R$. 
Indeed, for every $\ell \in S$, by assumption the class of $\theta^{g-1}/(g-1)!$ in $(CH^{g-1}(X_{\ov{k}})/\mathrm{num})_{\Q_\ell}$ belongs to $(CH^{g-1}(X)/\mathrm{num})_{\Z_\ell}$. 
Since $\Q\cap \Z_\ell=\Z_{(\ell)}$, this implies that the class of  $\theta^{g-1}/(g-1)!$ in $(CH^{g-1}(X_{\ov{k}})/\mathrm{num})_{\Q}$ comes from $(CH^{g-1}(X)/\mathrm{num})_{\Z_{(\ell)}}$.

$(3)\Longrightarrow (1)$. Observe that $X$ is a direct summand of $X\times\widehat{X}$.
Moreover, letting $\theta\in \on{NS}((X\times \widehat{X})_{\ov{k}})^G$ be the class of $\mathcal{P}$ on $X\times \widehat{X}$, we have $\deg(\theta^{2g}/(2g)!)=(-1)^{g}$; see \cite[Corollary 1 p. 129 and The Riemann--Roch Theorem p. 150]{mumford1970abelian}.
Now $(3)\Longrightarrow (1)$ follows from $(5)\Longrightarrow (1)$ applied to $\theta$.
\end{proof}

\begin{thm}\label{thm:main2-precise}
Let $k$ be a field of characteristic $p\geq 0$,
and let $X$ be an abelian variety over $k$, of dimension $g$. Let $S$ be a set of prime numbers and $R\coloneqq \cap_{\ell\in S}\Z_{(\ell)}$.

\medskip
When $k$ is finitely generated and $p\not\in S$, consider the following:
\begin{enumerate}
\item[(1$_{fg}$)] The abelian variety $X$ is a  prime-to-$S$ direct summand of a product of Jacobians of curves which contain a $k$-point.
\item[(2$_{fg}$)] The $\ell$-adic 
%(resp. crystalline) 
cycle map (\ref{ell-adic-cycle-map}) 
%(resp. (\ref{p-adic-cycle-map})) 
is surjective for all $\ell\in S$ 
%(resp. $\ell=p$ if $p\in S$) 
and each $i\in\left\{1,g-1\right\}$.
\item[(3$_{fg}$)] The $\ell$-adic
%(resp. crystalline) 
cycle map (\ref{ell-adic-cycle-map}) 
%(resp. (\ref{p-adic-cycle-map}))
is surjective for all $\ell\in S$ 
%(resp. $\ell=p$ if $p\in S$) 
and $i=g-1$.
\end{enumerate}
The statement (1$_{fg}$) implies (2$_{fg}$), (2$_{fg}$) implies (3$_{fg}$),
and if there exists $\theta\in \on{NS}(X_{\ov{k}})^G_R$ such that $\deg(\theta^g/g!)$ is prime to $S$, then (1$_{fg}$) - (3$_{fg}$) are all equivalent.

\medskip

In addition,
consider the following:
\begin{enumerate}
\item[(1$_{{fg}_s}$)]
The abelian variety $X_{k_s}$ is a  prime-to-$S$ direct summand of a product of Jacobians.
\item[(2$_{{fg}_s}$)] 
For every prime number $\ell\in S$, the cycle map
\[
CH^{g-1}(X_{k_s})_{\Z_\ell}\rightarrow H^{2g-2}(X_{\ov{k}},\Z_\ell(g-1))^{(1)}
\]
is surjective.
\end{enumerate}
Then (1$_{{fg}_s}$) implies (2$_{{fg}_s}$),
and if there exists $\theta\in \on{NS}(X_{\ov{k}})_R$ such that $\deg(\theta^g/g!)$ is prime to $S$, then (1$_{{fg}_s}$) and (2$_{{fg}_s}$) are equivalent.

\medskip
When $k=\F_q$ is finite, consider the following:
\begin{enumerate}
\item[(1$_f$)] The abelian variety $X$ is a  prime-to-$S$ direct summand of a product of Jacobians of curves which contain an $\F_q$-point.
\item[(2$_f$)] The $\ell$-adic (resp. crystalline) 
cycle map (\ref{ell-adic-cycle-map}) (resp. (\ref{p-adic-cycle-map}))
is surjective for all $\ell\in S\setminus\{p\}$ (resp. $\ell=p$ if $p\in S$) and $i=g-1$.
\end{enumerate}
Then (1$_{f}$) implies (2$_{f}$),
and if there exists $\theta\in \on{NS}(X_{\ov{\F}_q})^G_R$ such that $\deg(\theta^g/g!)$ is prime to $S$, then (1$_{f}$) and (2$_{f}$) are equivalent.

\medskip

When $k$ is the field $\R$ of real numbers, consider the following:
\begin{enumerate}
\item[(1$_r$)] The abelian variety $X$ is a  prime-to-$S$ direct summand of a product of Jacobians of curves which contain an $\R$-point.
\item[(2$_r$)] $CH^{g-1}(X)\rightarrow \on{Hdg}^{2g-2}(X_\C,\Z(g-1))^G$ has cokernel of order prime to $S$.
\end{enumerate}
Then (1$_r$) implies (2$_r$), and if there exists $\theta\in \on{NS}(X_\C)^G_R$ such that $\deg(\theta^g/g!)$ is prime to $S$, then (1$_r$) and (2$_r$) are equivalent.
\end{thm}

\begin{proof}
$(1_{fg})\Longrightarrow(2_{fg})$. It is enough to prove $(2_{fg})$ for $X\times \widehat{X}$: indeed, using the natural splitting of the projection $X\times \widehat{X}\rightarrow X$, this implies $(2_{fg})$ for $X$. 

If (1$_{fg}$) holds for $X$, then it holds for $X\times \widehat{X}$. By \Cref{lem:neron-severi}(2), the pullback map $\on{NS}(X\times\widehat{X})_R\to \on{NS}((X\times \widehat{X})_{k_s})^G_R$ is an isomorphism. 
Passing to $G$-invariants in the isomorphism (\ref{eq:tate-zarhin-faltings-1}), we deduce that for every $\ell\in S$, the cycle map \[CH^1(X\times \widehat{X})_{\Z_\ell}\rightarrow H^2((X\times \widehat{X})_{\ov{k}},\Z_\ell(1))^G\] is surjective. 
By \Cref{thm:main1-precise} (1)$\Rightarrow$(3), the classes of $c_1(\mathcal{P}_{X\times\widehat{X}})^{4g-2}/(4g-2)!$ in integral $\ell$-adic cohomology for all primes $\ell\in S$ are algebraic, hence represented by some $\Gamma_\ell \in CH^{4g-2}(X\times \widehat{X}\times \widehat{X}\times X)_{\Z_\ell}$. We have isomorphisms
\[
(\Gamma_\ell)_*\colon H^2((X\times \widehat{X})_{\ov{k}},\Z_\ell(1))^G\xrightarrow{\sim}H^{4g-2}((X\times\widehat{X})_{\ov{k}},\Z_\ell(2g-1))^G
\]
for all $\ell \in S$; see \cite[Proposition 1]{beauville1983fourier} or \cite[Section 7]{totaro2021integral}.
%when $\ell\neq p$, and \Cref{lem:p-adic-abelvar} (3) when $\ell=p$. 
By the compatibility of the cycle maps with the action of correspondences, we conclude that $(2_{fg})$ holds for $X\times \widehat{X}$, as desired.

$(2_{fg})\Longrightarrow(3_{fg})$. Obvious.

$(3_{fg})\Longrightarrow(1_{fg})$. This follows from \Cref{thm:main1-precise} (5)$\Rightarrow$(1).

The proofs of $(1_{{fg}_s})\Longleftrightarrow(2_{{fg}_s})$,
$(1_{f})\Longleftrightarrow(2_f)$,
and $(1_r)\Longleftrightarrow(2_r)$ are similar. 
For $(1_f)\Longleftrightarrow(2_f)$, we use the isomorphisms (\ref{eq:tate-zarhin-faltings-2}) and (\ref{eq:tate-zarhin-faltings-3}) instead of (\ref{eq:tate-zarhin-faltings-1}), as well as \Cref{lem:p-adic-abelvar} (3) in addition to \cite[Proposition 1]{beauville1983fourier} or \cite[Section 7]{totaro2021integral}.
For $(1_r)\Longleftrightarrow(2_r)$, we use the fact that for every abelian variety $X$ over the field $\R$ of real numbers, $CH^1(X)\rightarrow \on{Hdg}^2(X_\C,\Z(1))^G$ is surjective; see \cite[\S 5.2]{dGF2024integral}.
\end{proof}

\begin{proof}[Proof of \Cref{thm:main1}]
Follows from Theorems \ref{thm:main1-precise} and \ref{thm:main2-precise}, where $S$ is taken to be the set of all prime numbers.
\end{proof}

\begin{proof}[Proof of \Cref{cor:main}]
We recall that, by the Lang--Weil estimate, every geometrically integral $\F_q$-scheme of finite type contains a zero-cycle of degree $1$. Let $X$ be an abelian variety over $\F_q$, and let $\theta\in \on{NS}((X\times \widehat{X})_{\ov{\F}_q})^G$ be the class of $\mc{P}$ on $X\times \widehat{X}$. We have $\on{deg}(\theta^{2g}/(2g)!)=(-1)^g$; see \cite[Corollary 1 p. 129 and The Riemann--Roch Theorem p. 150]{mumford1970abelian}.

If $X$ is a direct summand of a product of Jacobians, then so is $X\times\widehat{X}$. Thus, by \Cref{thm:main1}$ (1')\Rightarrow (4)$, the integral Tate conjecture for $1$-cycles holds for $X\times\widehat{X}$, and hence for $X$.

Conversely, if the integral Tate conjecture for $1$-cycles holds for $X\times\widehat{X}$, then by \Cref{thm:main1}$ (4)\Rightarrow (1')$ we deduce that $X\times\widehat{X}$ is a direct summand of a product of Jacobians, and hence so is $X$.
\end{proof}

\begin{proof}[Proof of \Cref{thm:direct-summand-tate-consequence}]
Let $\F_q$ be a finite field of characteristic $p$, let $X$ be an abelian variety over $\F_q$, and let $g\coloneqq \dim X$. Since $X$ is a direct summand of $X\times\hat{X}$, replacing $X$ by $X\times \hat{X}$ we may assume that there exists $\theta\in \on{NS}(X_{\ov{\F}_q})^G$ such that $\on{deg}(\theta^{g}/g!)=\pm 1$. By a theorem of Schoen \cite[Theorem 0.5]{schoen1998integral}, the Tate conjecture for divisors on surfaces over finite fields of characteristic $p$ implies that the cycle map
\[CH^{g-1}(X_{\ov{\F}_q})_{\Z_\ell}\rightarrow H^{2g-2}(X_{\ov{\F}_q},\Z_\ell(g-1))^{(1)}\]
is surjective for every prime number $\ell$ different from $p$.
The proof then follows from \Cref{thm:main2-precise}  applied to the set $S$ of all prime numbers different from $p$.
\end{proof}

\section{Weil restrictions and the direct summand property}\label{section:weil-restrictions}

Let $k$ be a field, let $k'/k$ be a finite separable field extension, and let $X'$ be an abelian variety over $k'$. We write $X\coloneqq \on{Res}_{k'/k}(X')$ for the Weil restriction of $X'$ along $k'/k$, which is represented by a $k$-group by \cite[Proposition A.5.1]{conrad2015pseudo} (see also \cite[\S 7.6, Theorem 4]{bosch1990neron}). If $k'=k^n$ for some integer $n\geq 0$, then $X=\prod_{i=1}^nX_i$, where $X_i$ is  fiber of $X'\to \on{Spec}(k')$ over the $i$-th factor $k$. Moreover, the formation of $\on{Res}_{k'/k}(X')$ commutes with arbitrary field extensions of $k$. Therefore, letting $\Sigma\coloneqq \on{Hom}_k(k',k_s)$ be the finite $G$-set corresponding to $k'/k$, the $k_s$-algebra isomorphism \[k'\otimes_kk_s\xrightarrow{\sim} \prod_{\sigma\in \Sigma} k_s,\qquad a\otimes b\mapsto (\sigma(a)b)_{\sigma\in \Sigma}\] induces a $k_s$-group isomorphism
\begin{equation}\label{eq:res-splits}
X_{k_s}\xrightarrow{\sim} \prod_{\sigma \in \Sigma}(X')^{\sigma},
\end{equation}
where $(X')^\sigma$ is the base change of $X'$ along the homomorphism $\sigma\colon k'\to k_s$. In other words, $X$ is obtained from $\prod_{\sigma \in \Sigma}(X')^{\sigma}$ by Galois descent. In particular, $X$ is an abelian variety of dimension $|\Sigma|\on{dim}(X')$ over $k$. Here and in what follows, for every $\sigma \in \Sigma$, we let \[\on{pr}_\sigma\colon \prod_{\tau \in \Sigma}(X')^{\tau} \rightarrow (X')^\sigma\] be the projection corresponding to $\sigma$.

Let $\widehat{X}\coloneqq \on{Res}_{k'/k}(\widehat{X}')$. By Galois descent, the line bundle $\bigotimes_{\sigma \in \Sigma}\on{pr}_\sigma^*(\mathcal{P}_{X'})$ on $\prod_{\sigma \in \Sigma}(X'\times \widehat{X}')^\sigma$ descends to a line bundle $L$ on $X\times \widehat{X}$. (Since $(X\times \hat{X})(k)\neq\emptyset$, it suffices to check that the isomorphism class of the line bundle in question is Galois invariant.) 
 By \cite[\S 13, p. 123, Proposition]{mumford1970abelian}, $L$ is a Poincar\'e line bundle $\mc{P}_X$ on $X\times\widehat{X}$. In formulas, we have an isomorphism of line bundles on $X_{k_s}$:
\begin{equation}\label{eq:poincare-descent}
    (\mc{P}_X)_{k_s}\cong \bigotimes_{\sigma \in \Sigma}\on{pr}_\sigma^*(\mathcal{P}_{X'})
\end{equation}
In particular, as one would expect from the notation, $\widehat{X}$ is the dual abelian variety of $X$.

\begin{prop}\label{lem:weil-restrictions-precise}
Let $k$ be a field of characteristic $p\geq 0$, let $k'/k$ be a finite separable field extension, let $X'$ be an abelian variety over $k'$, and set $X\coloneqq \on{Res}_{k'/k}(X')$. Let $S$ be a set of prime numbers, and let $R\coloneqq \cap_{\ell\in S}\Z_{(\ell)}$. The equivalent statements (1) - (3) in \Cref{thm:main1-precise}, with respect to $S$, hold for $X'/k'$ if and only if they hold for $X/k$.
\end{prop}
\begin{proof}
If $X$ is a prime-to-$S$ direct summand of a product of Jacobians of smooth projective $k$-curves with a $k$-point, then $X_{k'}$ is a prime-to-$S$ direct summand of a product of Jacobians of smooth projective $k'$-curves with a $k'$-point. Since $k'$ is an \'etale $k$-algebra, the diagonal map $k'\to k'\otimes_kk'$ is a section of the multiplication map $k'\otimes_kk'\to k'$, and as a result the canonical $k'$-group inclusion $X'\to X_{k'}$ is split. We conclude that $X'$ is a prime-to-$S$ direct summand of a product of Jacobians of smooth projective $k'$-curves with a $k'$-point. This proves that if statement (1) of \Cref{thm:main1-precise} holds for $X/k$, then it holds for $X'/k'$.

Conversely, we assume that statement (3) of \Cref{thm:main1-precise} holds for $X'/k'$, and we prove that statement (1) of \Cref{thm:main1-precise} holds for $X/k$. We fix an inclusion of $k'$ in $k_s$, and we let $G'\coloneqq \on{Gal}(k_s/k')$ be the corresponding finite-index subgroup of $G$. Let $g$ be the dimension of $X'$, let $\ell\in S$, and define $A$ to be the class of $c_1(\mathcal{P}_{X'})^{2g-1}/(2g-1)!$ in integral $\ell$-adic cohomology if $\ell \neq p$ or in crystalline cohomology if $\ell=p$, and let $\alpha\in Z^{2g-1}(X'\times\widehat{X}')_{\Z_\ell}$ be a $1$-cycle whose cycle class is equal to $A$. 
Let $\beta\in Z^{2g}(X'\times \widehat{X}')_{\Z_\ell}$ be the zero-cycle of $0\in (X'\times\widehat{X}')(k')$, then the cycle class of $\beta$ is equal to
the class $B$ of $c_1(\mathcal{P}_{X'})^{2g}/(2g)!$ in integral $\ell$-adic cohomology if $\ell\neq p$ or in crystalline cohomology if $\ell=p$, because $\deg(\beta)=\deg(B)=1$ and the degree map 
gives an isomorphism $H^{4g}((X'\times \widehat{X}')_{\ov{k}},\Z_\ell(2g))\xrightarrow{\sim} \Z_\ell$ if $\ell \neq p$ or an isomorphism $H^{4g}((X\times\widehat{X})_{k_p}/W(k_p))^{F=p^{2g}}\xrightarrow{\sim} \Z_p$ if $\ell=p$. 

Consider the discrete $G$-set $\Sigma\coloneqq \on{Hom}_k(k',k_s)$, and let $N\coloneqq |\Sigma|=[k':k]$. We define a $1$-cycle on $\prod_{\sigma\in \Sigma}(X'\times \widehat{X}')^\sigma$ by the formula
\[
\gamma' \coloneqq \sum_{\sigma\in\Sigma}\left(\on{pr}_{\sigma}^*(\alpha^{\sigma})\prod_{\tau\in \Sigma\setminus\{\sigma\}}\on{pr}_{\tau}^*(\beta^{\tau})\right),
\]
where for all $\sigma\in \Sigma$ we let $\alpha^\sigma$ and $\beta^\sigma$ be the pullbacks of $\alpha$ and $\beta$ along $\sigma$, respectively.
Since $\gamma'$ is invariant under the $G$-action, there exists a unique $1$-cycle $\gamma\in Z^{2gN-1}(X\times \widehat{X})_{\Z_\ell}$ such that $\gamma_{k_s}=\gamma'$.
By (\ref{eq:poincare-descent}), the cycle class of $\gamma$ in integral $\ell$-adic cohomology if $\ell\neq p$ or in crystalline cohomology if $\ell=p$
is given by the class of $c_1(\mathcal{P}_X)^{2gN-1}/(2gN-1)!$, which is equal to
\[
\sum_{\sigma\in\Sigma}\left(\on{pr}_{\sigma}^*(A^{\sigma})\prod_{\tau\in\Sigma\setminus\{\sigma\}}\on{pr}_{\tau}^*(B^{\tau})\right).
\]
Therefore
(5) in \Cref{thm:main1-precise} holds for $X\times \widehat{X}$ and the class $\theta \in \on{NS}((X\times\widehat{X})_{\ov{k}})^G$ of $\mathcal{P}$ on $X\times \widehat{X}$, 
which implies that (1) in \Cref{thm:main1-precise} holds for $X\times \widehat{X}$, and hence also for $X$.
\end{proof}

\begin{proof}[Proof of \Cref{thm:tate-finite-wr-jacobians}]
By the Lang--Weil estimate, every geometrically integral variety over any finite field has a zero-cycle of degree $1$.
Thus, letting $S$ be the set of all prime numbers, \Cref{thm:main1-precise} (2) holds for $X'$,
hence by \Cref{lem:weil-restrictions-precise}, assertions (1) - (3) of \Cref{thm:main1-precise} hold for $\on{Res}_{\F_{q'}/\F_q}(X')$, so that by \Cref{thm:main2-precise}, the integral Tate conjecture for $1$-cycles holds for $\on{Res}_{\F_{q'}/\F_q}(X')$.
\end{proof}

\begin{proof}[Proof of \Cref{thm:tate-finite-Fp}]
Let $p$ be a prime, and let $\F_q$ be a finite field of characteristic $p$. 
\Cref{cor:main} shows that
the integral Tate conjecture holds for all abelian varieties over $\F_q$ if and only if every abelian variety over $\F_q$ is a direct summand of a product of Jacobians. By \Cref{lem:weil-restrictions-precise},
the latter is implied by the statement that every abelian variety over $\F_p$ is a direct summand of a product of Jacobians,
and by the Zarhin trick \cite{zarhin1974remark},
the assertion is further reduced to the statement that every principally polarized abelian variety over $\F_p$ is a direct summand of a product of Jacobians.
Finally, by \Cref{thm:main1}, the last statement is equivalent to the validity of the integral Tate conjecture for $1$-cycles on all principally polarized abelian varieties over $\F_p$.
\end{proof}

\section*{Acknowledgements}
We thank Marco D'Addezio for helpful correspondence and for pointing out a mistake in a previous version of \Cref{thm:tate-zarhin-faltings-dejong} over imperfect ground fields. The second-named author thanks Stefan Schreieder for interesting discussions and useful suggestions. The second-named author was supported by the European Union's Horizon 2020 research and innovation programme under grant agreement No. 948066 (ERC--StG RationAlgic).

\newcommand{\etalchar}[1]{$^{#1}$}

\end{document}